\documentclass[10pt]{amsart}
\usepackage{amsfonts,amssymb,stmaryrd,amscd,amsmath,latexsym,amsbsy,xypic}
\usepackage{graphicx,psfrag,upgreek}

\numberwithin{equation}{section} \theoremstyle{plain}
\newtheorem{thm}{Theorem}[section]

\newtheorem{prop}[thm]{Proposition}

\newtheorem{defi}[thm]{Definition}
\newtheorem{lem}[thm]{Lemma}
\newtheorem{cor}[thm]{Corollary}

\theoremstyle{remark}
\newtheorem{rema}[thm]{Remark}

\newcommand{\N}{\mathbb{N}}
\newcommand{\Z}{\mathbb{Z}}

\newcommand{\R}{\mathbb{R}}
\newcommand{\C}{\mathbb{C}}

\title[Bispectral quantum KZ equations]{Bispectral quantum
Knizhnik-Zamolodchikov equations for arbitrary root systems}
\author{Michel van Meer}
\address{Korteweg-de Vries Institute for Mathematics, University of Amsterdam,
Science Park 904, 1098 XH Amsterdam, The Netherlands.}
\email{m.vanmeer@uva.nl}

\begin{document}

\maketitle

\begin{abstract}
The bispectral quantum Knizhnik-Zamolodchikov (BqKZ) equation
corresponding to the affine Hecke algebra $H$ of type $A_{N-1}$ is
a consistent system of $q$-difference equations which in some
sense contains two families of Cherednik's quantum affine
Knizhnik-Zamolodchikov equations for meromorphic functions with
values in principal series representations of $H$. In
this paper we extend this construction of BqKZ to the case where
$H$ is the affine Hecke algebra associated to an arbitrary
irreducible reduced root system. We construct explicit solutions
of BqKZ and describe its correspondence to a bispectral problem
involving Macdonald's $q$-difference operators.
\end{abstract}


\section{Introduction}

The bispectral quantum Knizhnik-Zamolodchikov (BqKZ) equations
of type $\textup{GL}_N$ were introduced in \cite{MS}. The BqKZ equations
make up a consistent system of $q$-difference equations for
functions depending on two torus variables $t,\gamma\in
T:=(\C\setminus\{0\})^N$, such that for fixed $\gamma\in T$, the
equations in $t$ form Cherednik's \cite{CqKZ} quantum affine
Knizhnik-Zamolodchikov equations associated with the principal
series module $M_\gamma$ of the affine Hecke algebra $H$ of type $\textup{GL}_N$ with central character $\gamma$, while on
the other hand, for fixed $t\in T$, the equations in $\gamma$ form
another system of quantum affine KZ equations associated with
$M_{t^{-1}}$. This second system is expected to relate to Etingof
and Varchenko's system of dynamical $q$-difference equations
(\cite{EV}).

In the present paper, we extend the theory of BqKZ and its solutions
to arbitrary root systems. Apart from the case of $\textup{GL}_N$,
which was treated \cite{MS}, there are three cases to consider in the
Macdonald-Cherednik theory, namely the twisted and untwisted reduced
affine root systems and the nonreduced affine root system of type $C^\vee C$
(see \cite[(1.4.1)-(1.4.3)]{M}). In this paper we consider the twisted case
(\cite[(1.4.2)]{M}), the untwisted case is expected to allow for a similar treatment. The construction of BqKZ for $C^\vee C$ (along the lines of \cite{MS}) appeared in a recent preprint by Takeyama \cite{Ta}, so the picture is now rather complete.

Let us explain the ideas involved in a bit more detail. Choose
$0<q<1$. Let $W=W_0\ltimes P^\vee$ be the (extended)
affine Weyl group, the semidirect product of the finite Weyl group
$W_0$ and the coweight lattice $P^\vee$, corresponding to some
reduced irreducible root system of rank $N$. Consider the complex torus $T:=\textup{Hom}_\Z(P^\vee,\C^\times)$. Transposing the natural action of $W_0$ on $P^\vee$ gives rise to an action of $W_0$ on $T$. For $\lambda\in P^\vee$, let $q^\lambda\in T$ be defined by
\[
q^\lambda(\mu):=q^{\langle\lambda,\mu\rangle},\qquad{\mu\in P^\vee}.
\]
The action of $W_0$ on $T$ extends to an action of $W$ on $T$ by letting $\lambda\in P^\vee$ act via $t\mapsto q^\lambda t$. Let $V$ be a finite-dimensional complex vector space of dimension $\#W_0$. The BqKZ system which we will introduce, is a system of $q$-difference equations of the form
\[
C_{(\lambda,\mu)}(t,\gamma)f(q^{-\lambda} t,q^\mu\gamma)=
f(t,\gamma),\qquad\lambda,\mu\in P^\vee,
\]
for meromorphic functions $f$ on $T\times T$ with values in $V$.
Here $C_{(\lambda,\mu)}$ ($\lambda,\mu\in P^\vee$) are
$\textup{End}(V)$-valued meromorphic functions on $T\times T$,
satisfying the following cocycle property
\[
C_{(\lambda+\nu,\mu+\xi)}(t,\gamma)=C_{(\lambda,\mu)}(t,\gamma)
C_{(\nu,\xi)}(q^{-\lambda}t,q^{\mu}\gamma),\qquad\lambda,\mu,\nu,\xi\in P^\vee,
\]
which implies that BqKZ is a holonomic system of $q$-difference equations.

BqKZ contains, in some sense, two families of Cherednik's quantum
affine KZ equations associated with the principal series
representation of $H$. We recall that the quantum affine KZ
equations associated with a finite dimensional $H$-module $M$ is a
consistent system of $q$-difference equations of the form
\[
F^M_{\lambda}(t)f(q^{-\lambda}t)=f(t),\qquad \lambda\in P^\vee,
\]
for meromorphic functions $f$ on $T$ with values in $M$, and where
$F^M_\lambda$ ($\lambda\in P^\vee$) are $\textup{End}(M)$-valued
meromorphic functions on $T$ (see Subsection \ref{subsecqKZB}). Now the first family of quantum affine KZ
equations inside BqKZ is parametrized by
$\gamma\in T\simeq\{1\}\times T\subset T\times T$.
More precisely, if we fix $\gamma=\zeta\in T$, we have
\begin{equation*}
C_{(\lambda,e)}(t,\zeta)=F_\lambda^{M_\zeta}(t),
\end{equation*}
where $M_\zeta$ is the principal series representation of $H$ with
central character $\zeta$, which as a vector space can be
identified with $V$ via a $\zeta$-dependent isomorphism.
Similarly, interchanging the roles of the torus variables $t$ and
$\gamma$, BqKZ contains a second family of quantum affine KZ
equations, parametrized by $t\in T$ (related to the affine Hecke
algebra module $M_{t^{-1}}$).

Let us give a short overview of the paper. After the construction
of BqKZ we introduce the principal series representation, needed
to express the (asymptotic) values of the connection matrices
$C_{(\lambda,\mu)}(t,\gamma)$. These in turn are used to construct
an asymptotically free self-dual meromorphic solution $\Phi$ of
BqKZ. The set of solutions $\textup{SOL}$ of BqKZ allows an action
of $W_0$, and the orbit $W_0\Phi$ constitutes a basis of
$\textup{SOL}$ viewed as a vector space over the field of
$q$-dilation invariant meromorphic functions on $T\times T$.

For  $\textup{GL}_N$, a correspondence \cite[Thm. 6.16]{MS}
between solutions of BqKZ and solutions of a bispectral problem
involving Ruijsenaars' commuting trigonometric $q$-difference
operators (also known as Macdonald-Ruijsenaars operators) was
derived as a bispectral incarnation of Cherednik's \cite[Thm.
4.4]{Cref} embedding of the solutions of the quantum affine KZ
equations (for $\textup{GL}_N$) into the solutions of the
Ruijsenaars eigenvalue problem. The latter has been generalized to an embedding of the solution space of the quantum affine KZ equations for an arbitrary root system into the solution space of a system of $q$-difference equations involving Macdonald's $q$-difference operator (see \cite[Thm.4.6]{Ka} and \cite{CInd}). We give the analog of the bispectral correspondence \cite[Thm. 6.16]{MS} in the setting of arbitrary root systems.

As for $\textup{GL}_N$, we may apply the correspondence to $\Phi$
to obtain a self-dual Harish-Chandra series solution of the
bispectral problem. It is a bispectral analogue of (difference)
Harish-Chandra series solutions of the spectral problem for
Macdonald's $q$-difference operators, which were studied in
\cite{EK1} and \cite{KK} for root systems of type $A$ and in
\cite{LS} for arbitrary root systems. We will obtain new results
on the convergence and singularities of the Harish-Chandra series
from the corresponding results for $\Phi$.

Though the general constructions are more or less the same as for
$\textup{GL}_N$, various technical results require a different
approach. This becomes apparent in Section \ref{SectionFormalB}
when computing the cocycle values, in Section \ref{SectionSolB}
determining the asymptotic behavior of the $q$-connection matrices
and their singularities, and in Section \ref{SectionCorrespB}
finding the leading term of $\Phi$. An important difference with
the case of $\textup{GL}_N$, complicating some of the proofs, is
the fact that the affine Weyl group of type $\textup{GL}_N$ (and
the corresponding affine Hecke algebra) allows a rather convenient
presentation in terms of the finite Weyl group (respectively finite Hecke algebra) and an affine
Dynkin diagram automorphism (see \cite[Lemma 1.3.4]{C} or
\cite[\S2.1]{MS}), which is lacking for affine Weyl groups (respectively affine Hecke algebras) of arbitrary type. In this paper we give all the main constructions and provide those proofs that are substantially different from the proofs for $\textup{GL}_N$.

\subsection*{Conventions}\hspace{1cm}\\
{\bf --} $\otimes$ always stands for tensor product over $\C$ and
$\textup{End}(M)$, for a module $M$ over $\C$,
stands for $\C$-linear endomorphisms.\\
{\bf --} $\mathbb{N}=\{1,2,\ldots\}$.\\
{\bf --} For a module $M$ over a commutative ring $R$ and a ring
extension $R\subset S$, we write $M^S=S\otimes_RM$.\\
{\bf --} For $a,r\in\R$ with $a>0$, we choose $a^r$ to be the
positive real branch of the power function.

\subsection*{Acknowledgments} The author is supported by the
Netherlands Organization for Scientific Research (NWO) in the
VIDI-project ``Symmetry and modularity in exactly solvable
models''. He likes to thank Jasper Stokman for his advice and many
valuable discussions.

\section{Notations}

\subsection{Root data}
Let $(V,\langle\:,\:\rangle)$ be a real Euclidean space of
dimension $N>0$. Let $\widehat{V}$ be the space of affine linear
real functions on $V$. Consider the 1-dimensional vector space $\R
c$. There is a natural isomorphism of real vector spaces
$V\oplus\mathbb{R} c\simeq\widehat{V}$ via
$v+rc\mapsto(u\mapsto\langle v,u\rangle+r)$ for $u,v\in V$ and
$r\in\mathbb{R}$. We will use this isomorphism to identify
$\widehat{V}$ and $V\oplus\mathbb{R}c$, thus regarding
$c\in\widehat{V}$ as the constant function equal to 1.

The map $D\colon\widehat{V}\to V$ defined by $D(v+rc)=v$
($v\in\mathbb{R}$, $r\in\mathbb{R}$) is called the gradient map.
We extend the inner product $\langle\:,\:\rangle$ to a positive
semi-definite bilinear form on $\widehat{V}$ by
\[\langle f,g\rangle:=\langle Df,Dg\rangle,\]
for $f,g\in\widehat{V}$. For $f\in\widehat{V}$ with $Df\neq0$, we
set $f^\vee:=2f/\langle f,f\rangle\in\widehat{V}$.

Let $R\subset V$ be a reduced irreducible finite root system in
$V$ and assume that the scalar product is normalized such that
long roots have squared length 2. The Weyl group $W_0\subset O(V)$
associated to $R$ is the group generated by the orthogonal
reflections $s_\alpha$ in the hyperplanes $\alpha^{\bot}$
($\alpha\in R$). Explicitly, we have
\[
s_\alpha(v)=v-\langle v,\alpha\rangle\alpha^\vee,
\]
for $\alpha\in R$, $v\in V$. Fix a basis of simple roots
$\{\alpha_1,\ldots,\alpha_N\}$ of $R$. Write $R_+$ for the set of
positive roots, $R_-:=-R_+$ for the set of negative roots, and
$\phi$ for the highest root with respect to this basis. Note that
$\phi\in R_+$ is a long root (and so $\phi^\vee=\phi$).

We use the standard notations for the (co)root and (co)weight
lattices, that is,
\[
\begin{split}
    Q&:=\mathbb{Z}\textnormal{-span of } R,\\
    Q^\vee&:=\mathbb{Z}\textnormal{-span of } R^\vee,\\
    P&:=\{\lambda\in V\mid
    \langle\lambda,\alpha^\vee\rangle\in\mathbb{Z},\:\forall\alpha\in R\},\\
    P^\vee&:=\{\mu\in V\mid
    \langle\mu,\alpha\rangle\in\mathbb{Z},\:\forall\alpha\in R\}.
\end{split}
\]
Note that $Q\subseteq P$ and $Q^\vee\subseteq P^\vee$.
Furthermore, since $\|\alpha\|^2=2$ for $\alpha\in R$ a long root
and thus $\|\alpha\|^2\in\{1,2/3\}$ for $\alpha\in R$ short, we
have
$\alpha^\vee=\frac{2}{\|\alpha\|}\alpha\in\{\alpha,2\alpha,3\alpha\}\subset
Q$ for any $\alpha\in R$. Hence $Q^\vee\subseteq Q$ and therefore
also $P^\vee\subseteq P$.

Let $L\subset V$ be any $W_0$-invariant lattice. The canonical
action of $W_0$ on $V$ extends to a faithful action of the
semi-direct product group $W_L:=W_0\ltimes L$ on $V$ such that
elements of $L$ act as translations. If we want to stress that we
view $\lambda\in L$ as an element of $W_L$, we write
$\textup{t}(\lambda)$. In this notation, $L\subset W_L$ acts on
$V$ by
$$\textup{t}(\lambda)v=v+\lambda,$$
for $\lambda\in L$ and $v\in V$. Transposing the action of $W_L$
on $V$ gives an action of $W_L$ on $\widehat{V}$. It is given by
\[
\begin{split}
w(v+rc)&=w(v)+rc,\qquad\quad w\in W_0,\\
\textup{t}(\lambda)(v+rc)&=v+(r-\langle v,\lambda\rangle)c,\qquad
\lambda\in L,
\end{split}
\]
for $v\in V$, $r\in\R$. Note that $\langle
w(f),w(g)\rangle=\langle f,g\rangle$ for all $f,g\in\widehat{V}$
and $w\in W_L$. In the case that $L=Q^\vee$,
$W_L=W_{Q^\vee}=W_0\ltimes Q^\vee$ is the affine Weyl group. The
extended affine Weyl group is $W_{P^\vee}=W_0\ltimes P^\vee$ and
we will simply denote it by $W$.

Associated to the reduced irreducible finite root system $R$ there
is a reduced irreducible affine root system
$S=S(R):=\{\alpha+rc\mid\alpha\in R, r\in\Z\}$ in $\widehat{V}$.
For $a\in S$, let $s_a\colon V\to V$ be the reflection in the
hyperplane $a^{-1}(\{0\})$, given by
\[
s_a(v)=v-a(v)Da^\vee,
\]
for $v\in V$. Then $s_a=s_{Da}\textup{t}(a(0)Da^\vee)\in
W_{Q^\vee}$. Note that $S\subset\widehat{V}$ is $W$-invariant. We
define an ordered basis $(a_0,\ldots,a_N)$ of $S$ by setting
\[
(a_0,a_1,\ldots,a_N):=(-\phi+c,\alpha_1,\ldots,\alpha_N).
\]
Write $S_+$ and $S_{-}$ for the associated sets of positive and
negative affine roots respectively. Note that
$$S_+:=\{\alpha+rc\mid\alpha\in R, r\geq\chi(\alpha)\},$$ where
$\chi$ is the characteristic function of $R_-$, i.e.,
$\chi(\alpha)=1$ if $\alpha\in R_-$, and $\chi(\alpha)=0$ if
$\alpha\in R_+$.

We put $s_i:=s_{a_i}\in W_{Q^\vee}\subseteq W$ for $i=0,\dots,N$.
The affine Weyl group $W_{Q^\vee}$ is a Coxeter group with Coxeter
generators the simple reflections $s_i$. For $w\in W$ write
$S(w):=S_+\cap w^{-1}S_-$. The length function $\ell$ on $W$ is
defined by
$$\ell(w):=\#S(w),\qquad w\in W.$$
The unique element with maximal length in $W_0$ is denoted by
$w_0$.

The finite abelian subgroup $\Omega:=\{w\in W\mid\ell(w)=0\}$ of
$W$ is isomorphic to $P^\vee/Q^\vee$ and we have
$$W\simeq W_{Q^\vee}\rtimes\Omega.$$
The action of $\Omega$ on $\widehat{V}$ restricts to a faithful
action on the set $\{a_0,\ldots,a_N\}$ of simple roots of $S$, so
we can view $\Omega$ as a group of permutations on the set of
indices $\{0,\ldots,N\}$. We write $\mathbb{C}[\Omega]$ for the
group algebra of $\Omega$.

The Bruhat order $\leq$ on $W_{Q^\vee}$ extends to a partial order
on $W$, referred to as the Bruhat order on $W$ (cf.
\cite[\S2.3]{M}). It is defined as follows. For $w=\omega u$ and
$w'=\omega' u'$ with $\omega,\omega'\in\Omega$ and $u,u'\in
W_{Q^\vee}$ we have by definition
\begin{equation}\label{BruhatOnWB}
w\leq w'\:\Longleftrightarrow\: \omega=\omega'\mbox{ and }u\leq
u'.
\end{equation}

\subsection{Algebra of $q$-difference reflection
operators}\label{subsecAlgofqDiffB} 

Consider the complex torus
$T:=\textup{Hom}_\Z(P^\vee,\C^\times)$. By transposition, the natural action of $W_0$ on $P^\vee$ gives rise to an action of $W_0$ on $T$. Fix $0<q<1$. For $\lambda\in P^\vee$, let $q^\lambda\in T$ be defined by
\[
q^\lambda(\mu):=q^{\langle\lambda,\mu\rangle},\qquad{\mu\in P^\vee}.
\]
The action of $W_0$ on $T$ extends to an action
of $W=W_0\ltimes P^\vee$ on $T$ by letting $\lambda\in P^\vee$ act via $t\mapsto q^\lambda t$. Let the evaluation of $t\in T$ in a point $\lambda\in P^\vee$ be denoted by $t^\lambda\in\C^\times$. Then, summarizing, we have an action of $W$ on $T$ given by
\[
\begin{split}
    (wt)^\mu&=t^{w^{-1}\mu},\\
    (\textup{t}(\lambda) t)^\mu&=q^{\langle\lambda,\mu\rangle}
    t^\mu,
\end{split}
\]
for $t\in T$, $w\in W_0$ and $\lambda,\mu\in P^\vee$. 

Let $\{\varpi_i^{\vee}\}_{i=1}^N$ be the set of fundamental coweights in
$P^\vee$ with respect to $\{\alpha_j\}_{=1}^N$, so $\langle\varpi_i^\vee,\alpha_j\rangle=\delta_{ij}$ for $1\leq i,j\leq N$.
We identify $T\simeq(\C\setminus\{0\})^N$ via
$t\leftrightarrow (t_1,\ldots,t_N)$ defined by
\[
t_i:=t^{\varpi^\vee_i}
\]
for $i=1,\ldots,N$. Under this identification, the action of $P^\vee$ on $T$ reads
\begin{equation}\label{lambdaActionB}
\textup{t}(\lambda)t=q^\lambda t=(q^{\langle\lambda,\varpi^\vee_1\rangle}t_1,\ldots,
q^{\langle\lambda,\varpi^\vee_N\rangle}t_N)
\end{equation}
for $\lambda\in P^\vee$ and $t=(t_1,\ldots,t_N)\in T$.

The algebra of complex-valued regular functions on $T$ is
$\C[x_1^{\pm1},\ldots,x_N^{\pm1}]=\textup{span}_\C\{x^\lambda\}_{\lambda\in
P^\vee}$, where $x_i$ is the coordinate function $x_i(t):=t^{\varpi^\vee_i}$ ($i=1,\ldots,N$) and $x^\lambda(t):=t^\lambda$ for $\lambda\in P^\vee$. Clearly, it is isomorphic to the group algebra $\C[P^\vee]$ of $P^\vee$. We write $\C[T]=\C[x_1^{\pm1},\ldots,x_N^{\pm1}]$ and we let
$\C(T)$ denote the field of rational functions on $T$,
$\mathcal{O}(T)$ the ring of analytic functions on $T$, and
$\mathcal{M}(T)$ the field of meromorphic functions on $T$. The
$W$-action on $T$ gives rise to a $W$-action by algebra
automorphisms on each of these function algebras, via
\[
(wf)(t)=f(w^{-1}t),
\]
for $w\in W$, $t\in T$ and $f$ a (regular, rational or
meromorphic) function on $T$. Note that for $\lambda\in P^\vee$
and $r\in\mathbb{R}$, we have
\[
w(x^{\lambda+rc})=x^{w(\lambda+rc)},
\]
where $x^{\lambda+rc}:=q^rx^\lambda\in\C[T]$.

By means of this $W$-action by field automorphisms on $\C(T)$, we
can form the smash product algebra $\mathbb{C}(T)\#_qW$, which we
call the algebra of $q$-difference reflection operators with
coefficients in $\C(T)$, since it acts canonically on $\C(T)$ and
$\mathcal{M}(T)$ as $q$-difference reflection operators. For
$f\in\C(T)$ we will write $f(X)\in\C(T)\#_q W$ for the operator on
$\mathcal{M}(T)$ (or $\C(T)$) defined as multiplication by $f$. We
will also write $X^{\lambda+rc}=q^rX^\lambda$ for $\lambda\in
P^\vee$ and $r\in \R$.
\begin{rema}\label{qDependenceB}
Note that since
$(\textup{t}(\lambda)f)(t)=f(q^{-\langle\varpi_1^\vee,\lambda\rangle}
t_1,\ldots,q^{-\langle\varpi_N^\vee,\lambda\rangle}t_N)$
($\lambda\in P^\vee$, $f\in\mathcal{M}(T)$), $\mathbb{C}(T)\#_qW$
actually depends on a choice for $q^{\frac{1}{m}}$, where $m\in\N$
is determined by $m\langle P^\vee,P^\vee\rangle=\Z$. Our global
convention concerning real powers of positive real numbers
justifies the apparent abuse of notation writing $q$ instead of
$q^{1/m}$.
\end{rema}

\subsection{The extended affine Hecke algebra and Cherednik's basic
representation}
Let $k_i$ ($i=0,\ldots,N$) be nonzero complex numbers such that
$k_i=k_j$ if $s_i$ and $s_j$ are conjugate in $W$. Write
$\underline{k}$ for the corresponding multiplicity label
$\underline{k}\colon S\to\C\setminus\{0\}$, so
$\underline{k}(a)=k_i$ for all $a\in W(a_i)$ ($i=0,\ldots,N$). We
set $k_a:=\underline{k}(a)$ for $a\in S$. Furthermore, for $w\in
W$ we define
\[
k(w):=\prod_{a\in S(w)}k_a.
\]
A coweight $\lambda\in P^\vee$ is called dominant if $\langle
\lambda,\alpha_i\rangle\geq0$ for $i=1,\ldots,N$. Let $P^\vee_{+}$
denote the set of dominant coweights.
\begin{lem}
For $\lambda\in P_+^\vee$, we have
\begin{equation}\label{klambdaB}
k(\textup{t}(\lambda))=\prod_{\alpha\in
R_+}k_\alpha^{\langle\lambda,\alpha\rangle}=\delta_{\underline{k}}^\lambda,
\end{equation}
where $\delta_{\underline{k}}\in T$ is defined by
$(\delta_{\underline{k}})_i=\prod_{\alpha\in
R_+}k_\alpha^{\langle\varpi^\vee_i,\alpha\rangle}$
($i=1,\ldots,N$).
\end{lem}
\begin{proof}
For $\lambda\in P_+^\vee$ we have
\[S(\textup{t}(\lambda))=\{\alpha+rc\mid \alpha\in R_+,\:0\leq
r<\langle\lambda,\alpha\rangle\},
\]
cf. \cite[\S2.4]{M}. Note that $k_{\alpha+rc}=k_\alpha$ for
$\alpha\in R$ and $r\in\Z$ since $\alpha+rc$ and $\alpha$ are
conjugate under the action of $W$. Indeed, for $\mu\in P^\vee$ we
have
$\textup{t}(\mu)(\alpha+rc)=\alpha+(r-\langle\mu,\alpha\rangle)c$
and for any $\alpha\in R$ there exists some $\nu\in P^\vee$ such
that $\langle\nu,\alpha\rangle=1$, so that we can take $\mu=r\nu$.
Therefore
\[
k(\textup{t}(\lambda))=\prod_{\stackrel{\stackrel{\alpha\in
R_+}{0\leq
r<\langle\lambda,\alpha\rangle}}{}}k_{\alpha+rc}=\prod_{\alpha\in
R_+}k_\alpha^{\langle\lambda,\alpha\rangle}.
\]
The second equality in \eqref{klambdaB} follows from the
definitions.
\end{proof}

\begin{defi}\label{defAHAB}
    The affine Hecke algebra $H_{Q^\vee}$ associated to the Coxeter system
    $(W_{Q^\vee},\{s_0,\ldots,s_N\})$ and the multiplicity label
    $\underline{k}$, is the unital complex associative algebra
    generated by elements $T_0,\ldots,T_N$, such that
    \begin{itemize}
        \item[\bf{(i)}] $T_0,\ldots,T_N$ satisfy the braid
        relations, i.e. if for $i\neq j$, we have
        $$s_is_js_i\cdots=s_js_is_j\cdots,$$ with $m_{ij}$ factors on each
        side, then
        $$T_iT_jT_i\cdots=T_jT_iT_j\cdots,$$
        with $m_{ij}$ factors on each side;
        \item[\bf{(ii)}] $(T_j-k_j)(T_j+k_j^{-1})=0,$  \:for
        $j=0,\ldots,N$.
    \end{itemize}
\end{defi}

Note that since $\underline{k}$ is $W$-invariant, the group
$\Omega$ acts on $H_{Q^\vee}$ by algebra automorphisms via
$T_i\mapsto T_{\omega(i)}$ for $i=0,\ldots,N$.

\begin{defi}
    The extended affine Hecke algebra $H=H(\underline{k})$ is
    the smash product $H:=H_{Q^\vee}\#\Omega$.
\end{defi}
For $w\in W$ and a reduced expression $w=\omega s_{i_1}\cdots
s_{i_{\ell(w)}}$ with $\omega\in\Omega$ and $i_k\in
\{0,\ldots,N\}$, we define
$$T_w:=\omega T_{i_1}\cdots
T_{i_{\ell(w)}}\in H,$$ which is independent of the reduced
expression chosen. The set $\{T_w\mid w\in W\}$ is a linear basis
of $H$. Note that for $\underline{k}\equiv1$ the extended affine
Hecke algebra is just the group algebra $\C[W]$ of $W$. The finite
Hecke algebra is the subalgebra $H_0=H_0(\underline{k})$ of $H$,
generated by $T_1,\ldots,T_N$.

For $\lambda\in P^\vee_{+}$, put
\[
Y^{\lambda}:=T_{\textup{t}(\lambda)}\in H,
\]
and for arbitrary $\lambda\in P^\vee$ put
\[
Y^{\lambda}:=Y^{\mu}(Y^{\nu})^{-1},
\]
if $\lambda=\mu-\nu$ with $\mu,\nu\in P^\vee_+$. Then the
$Y^{\lambda}$ ($\lambda\in P^\vee$) are well-defined and we have
$Y^0=1$ and
$Y^{\lambda}Y^{\mu}=Y^{\lambda+\mu}=Y^{\mu}Y^{\lambda}$ for all
$\lambda,\mu\in P^\vee$. Set $Y_i:=Y^{\varpi^\vee_i}$ for
$i=1,\ldots,N$.

For $\kappa\in\C\setminus\{0\}$ we define the functions
$b(z,\kappa)$ and $c(z,\kappa)$ by
\[\begin{split}
    b(z;\kappa)&:=\frac{\kappa-\kappa^{-1}}
    {1-z},\\
    c(z;\kappa)&:=
    \frac{\kappa^{-1}-\kappa z}
    {1-z},
  \end{split}
\]
as rational functions in $z$. Then for $a\in S$, we define
$b_{a;\underline{k},q}=b_a\in\C(T)$ and
$c_{a;\underline{k},q}=c_a\in\C(T)$ by
\[\begin{split}
    b_a(t)&:=b(t^{a^\vee};k_a)\\
    c_a(t)&:=c(t^{a^\vee};k_a).
  \end{split}
\]
\begin{rema}
The $q$-dependence of $b_{a,\underline{k},q}$ and
$c_{a;\underline{k},q}$ comes from the convention
$t^{\alpha+rc}=q^rt^\alpha$ for $\alpha\in R$ and $r\in\R$.
Note that
\begin{equation}\label{cParamRelB}
c_{a;\underline{k},q}(t^{-1})=c_{a;\underline{k}^{-1},q^{-1}}(t)
\end{equation}
for all $a\in S$ and $t\in T$. We leave out the subscripts
$\underline{k}$ and $q$ as long as there is no chance of confusion
(which is until Section \ref{SectionCorrespB}).
\end{rema}

Note that $b_a(t)=k_a-c_a(t)$ and $(wc_a)(t)=c_{w(a)}(t)$ for all
$w\in W$. It is convenient to introduce the notations
$b_j:=b_{a_j}$ and $c_j:=c_{a_j}$ for $j=0,\ldots,N$. The
following characterization of $H$ is due to Bernstein and
Zelevinsky (see, e.g., \cite[\S4.2]{M}).
\begin{thm}\label{HcharB}
The affine Hecke algebra $H=H(\underline{k})$ is the unique complex
associative algebra, such that\\
{\bf(i)} $H_0\otimes\C[T]\simeq H$ as complex vector spaces, via
$h\otimes f\mapsto hf(Y)$ for $h\in H_0$, $f\in\C[T]$, where
$f(Y)=\sum_\lambda a_\lambda Y^{\lambda}$ if $f=\sum_\lambda
a_\lambda x^\lambda\in\C[T]$;\\
{\bf(ii)} the canonical maps $H_0,\C[T]\hookrightarrow H$ are
algebra embeddings; we write
$\C_Y[T]=\textup{span}_\C\{Y^\lambda\}_{\lambda\in P^\vee}$ for
the
image of $\C[T]$ in $H$;\\
{\bf(iii)} Lusztig's relations are satisfied, that is,
\begin{equation}\label{eqLusztigB}
      f(Y)T_j=T_j(s_jf)(Y)+b_j(Y^{-1})\left(f(Y)-(s_jf)(Y)\right)
\end{equation}
for $j=1,\ldots,N$ and $f\in\C[T]$.
\end{thm}

\begin{rema}
Note that
$b_j(Y^{-1})\left(f(Y)-(s_jf)(Y)\right)\in\C_Y[T]$ although
$b_j(Y^{-1})$ by itself is not defined as an element of
$H$.
\end{rema}
We end this section with the definition of the double affine Hecke
algebra and state some of its key results. All of this is due to
Cherednik; see \cite{C}. It starts with the realization of the
affine Hecke algebra inside the algebra $\C(T)\#_q W$ of
$q$-difference reflection operators.

\begin{thm}\label{CherednikB}
There is a unique injective algebra homomorphism
$\rho=\rho_{\underline{k},q}\colon H\to\C(T)\#_q W$ satisfying
\[\begin{split}
\rho(T_i)&=k_i+c_i(X)(s_i-1),\qquad i=0,\ldots,N,\\
\rho(\omega)&=\omega,\hspace{3.5cm}\omega\in\Omega.
\end{split}\]
\end{thm}

\begin{rema}
The image $\rho(H)$ preserves $\C[T]$, viewed as a subspace of the
canonical $\C(T)\#_qW$-module $\C(T)$. The resulting faithful
representation of $H$ on $\C[T]$ is called the basic
representation of $H$.
\end{rema}

\begin{defi}
The double affine Hecke algebra
$\mathbb{H}=\mathbb{H}(\underline{k},q)$ is the subalgebra of
$\C(T)\#_q W$ generated by $H$ (i.e. by
$\rho_{\underline{k},q}(H)$) and by the multiplication operators
$f(X)$ ($f\in\C[T]$).
\end{defi}

\begin{rema}
Note that $\rho=\rho_{\underline{k},q}$ and
$\mathbb{H}=\mathbb{H}(\underline{k},q)$ actually depend on
$q^{\frac{1}{m}}$ (see Remark \ref{qDependenceB}).
\end{rema}

We view $\mathbb{H}$ as a left $\C[T]$-module by $(f,h)\mapsto
f(X)h$ ($f\in\C[T]$, $h\in \mathbb{H}$). The rule $f\otimes
h\mapsto f(X)h$ ($h\in H$, $f\in\C[T]$) induces an isomorphism of
$\C[T]$-modules
\begin{equation}\label{linisomHHB}
\C[T]\otimes H\simeq\mathbb{H},
\end{equation}
Similarly to Theorem \ref{HcharB}, the algebra structure of
$\mathbb{H}$ can be described in terms of the left-hand side of
\eqref{linisomHHB}, allowing for an abstract definition of
$\mathbb{H}$:
\begin{thm}
The double affine Hecke algebra $\mathbb{H}$ can be
characterized as the unique associative algebra satisfying\\
{\bf(i)} $\C[T]\otimes H\simeq \mathbb{H}$ as complex vector spaces;\\
{\bf(ii)} the canonical maps $H,\C[T]\hookrightarrow\mathbb{H}$ are
algebra embeddings;\\
{\bf(iii)} the following cross relations are satisfied: for
$f\in\C[T]$
\begin{eqnarray}\label{XcrossB}
    T_jf(X)\hspace{-0.25cm}&=&\hspace{-0.25cm}(s_jf)(X)T_j+
    b_j(X)\left(f(X)-(s_jf)(X)\right),
    \quad j=0,\ldots,N,\\
    \omega f(X)\hspace{-0.25cm}&=&\hspace{-0.25cm}(\omega
    f)(X)\omega,\quad \omega\in\Omega.\label{XcrossOmegaB}
\end{eqnarray}
\end{thm}

A crucial ingredient in the construction of the bispectral quantum
KZ equations is Cherednik's duality anti-involution on
$\mathbb{H}$ (see \cite[Thm. 1.4.8]{C}).

\begin{thm}\label{dualityInvoB}
There exists a unique anti-algebra involution $*\colon
\mathbb{H}\rightarrow \mathbb{H}$ determined by
\begin{equation*}
\begin{split}
T_w^*&=T_{w^{-1}},\qquad\: w\in W_0,\\
(Y^{\lambda})^*&=X^{-\lambda},\qquad \lambda\in P^\vee,\\
(X^{\lambda})^*&=Y^{-\lambda},\qquad \lambda\in P^\vee.
\end{split}
\end{equation*}
\end{thm}

\section{Bispectral quantum KZ equations}\label{SectionBqKZB}

In this section we extend the construction of the bispectral
quantum Knizhnik-Zamolodchikov equations for $\textup{GL}_N$
\cite{MS} to arbitrary root systems. First we recall Cherednik's
construction of the quantum affine Knizhnik-Zamolodchikov equations
\cite{CqKZ} associated with a finite-dimensional $H$-module.

\subsection{The quantum affine KZ equations}\label{subsecqKZB}

In order to define the quantum KZ equations we first need to
left-localize the double affine Hecke algebra $\mathbb{H}\simeq
\C[T]\otimes H$ (see Theorem \ref{linisomHHB}) with respect to
$\C[T]\setminus\{0\}$. As a complex vector space, the resulting
algebra $\widehat{\mathbb{H}}$ can be defined as
$\widehat{\mathbb{H}}\simeq \C(T)\otimes H$ and then its algebra
structure is determined by requiring $\C(T)$ and $H$ to be
subalgebras, and by requiring the cross relations \eqref{XcrossB}
and \eqref{XcrossOmegaB} to hold for $f\in\C(T)$.

The injective map $\rho$ of Theorem \ref{CherednikB} extends to an
injective algebra homomorphism
\[
\rho\colon\widehat{\mathbb{H}}\to\C(T)\#_q W
\]
by setting $\rho(f(X))=f(X)$ for $f\in\C(T)$. Note that
$\rho\big(c_j(X)^{-1}(T_j-b_j(X))\big)=s_j$
for $0\leq j\leq N$, which shows that $\rho$ is surjective and
therefore establishes an isomorphism
$\widehat{\mathbb{H}}\simeq\C(T)\#_q W$. Restricting the inverse
$\rho^{-1}$ to $W$ gives a realization of $W$ inside
$\widehat{\mathbb{H}}^\times$.

The left multiplication map turns $H$ into a left module over
itself. The action of $\widehat{\mathbb{H}}$ on the induced module
$\textnormal{Ind}_H^{\widehat{\mathbb{H}}}(H)=\widehat{\mathbb{H}}\otimes_H
H$ can be pushed forward along the linear isomorphism
$\widehat{\mathbb{H}}\otimes_H H\simeq\C(T)\otimes H$ to obtain an
algebra homomorphism
\[
\pi=\pi_{\underline{k}}\colon\widehat{\mathbb{H}}\to\textup{End}(\C(T)\otimes
H).
\]
We regard $\C(T)\#_qW\otimes H$ as a subalgebra of
$\textup{End}(\C(T)\otimes H)$ by letting $\C(T)\#_q W$ act on
$\C(T)$ as in Subsection \ref{subsecAlgofqDiffB}, and $H$ on $H$
by left multiplication. Then the pullback
$\tau_x=\tau_{x,\underline{k}}:=\pi\circ\rho^{-1}$ of $\pi$ along
$\rho^{-1}$ is an algebra homomorphism
\[
\tau_x\colon\C(T)\#_q W\to\C(T)\#_qW\otimes
H\subset\textup{End}(\C(T)\otimes H),
\]
which is explicitly given by
\[
\begin{split}
\tau_x(f)&=f(X)\otimes 1,\qquad\qquad f\in\C(T),\\
\tau_x(s_j)&=(c_j(X)^{-1}\otimes1)(s_j\otimes
T_j-b_j(X)s_j\otimes 1),\qquad 0\leq j\leq N,\\
\tau_x(\omega)&=\omega\otimes\omega,\qquad \omega\in\Omega,
\end{split}
\]
as can be verified by a direct computation using the formula for
$\rho^{-1}$ and the cross relations \eqref{XcrossB}.

\begin{rema}
The reason for the subscript $x$ in $\tau_x$ will become apparent
in the next subsection when we discuss the bispectral story. Then
two copies of $T$ will play a role and $x$ will denote the
coordinate functions on one of them.
\end{rema}

Note that $\tau_x(s_j)=F_{s_j}(X)(s_j\otimes 1)$ with
$F_{s_j}(X)=(c_j(X)^{-1}\otimes 1)(1\otimes
T_j-b_j(X)\otimes 1)\in\C(T)\otimes H$ and trivially
also $\tau_x(\omega)=F_\omega(X)(\omega\otimes1)$ with
$F_\omega=1\otimes\omega\in\C(T)\otimes H$. In fact, more
generally, we have
\[
\tau_x(w)=F_w(X)(w\otimes1),\qquad w\in W,
\]
where $F_w$ are $H$-valued rational functions on $T$ satisfying
\begin{equation}\label{qKZcocycleB}
F_{e}(t)=1,\quad F_{vw}(t)=F_v(t)F_{w}(v^{-1}t)
\end{equation}
for all $v,w\in W$ and $t\in T$. Viewed as elements of
$\textup{End}(\C(T)\otimes H)$ the $F_w(X)$ ($w\in W$) are
$\C(T)$-linear and invertible (indeed $F^{-1}_w(X)=(w^{-1}\otimes
1)\tau_x(w^{-1})$). In the language of non-abelian group
cohomology, \eqref{qKZcocycleB} means that $w\mapsto F_w(X)$
constitutes a cocycle $W\to\textup{GL}_{\C(T)}(\C(T)\otimes H)$,
where $\textup{GL}_{\C(T)}(\C(T)\otimes H)$ is a $W$-group via the
usual action of $W$ on the first tensor leg of
$\C(T)\otimes\textup{End}(H)\simeq\textup{GL}_{\C(T)}(\C(T)\otimes
H)$.

Now let $M$ be a left module over the affine Hecke algebra $H$.
Then $M^{\mathcal{M}(T)}=\mathcal{M}(T)\otimes M$ is a module over
$\C(T)\#_qW\otimes H$, where $\C(T)\#_qW$ acts on $\mathcal{M}(T)$
as described in subsection \ref{subsecAlgofqDiffB}. Consequently,
$\tau_x$ gives rise to a representation
\[
\tau_x^M\colon W\to\textup{GL}(M^{\mathcal{M}(T)}),
\]
defining $\tau^M_x(w)$ ($w\in W$) to be
$\tau_x(w)\in\C(T)\#_qW\otimes H$ acting on $M^{\mathcal{M}(T)}$.
Let $F^M_w$ ($w\in W$) denote the corresponding functions
$F_w\in\C(T)\otimes H$ acting on $M^{\mathcal{M}(T)}$. For
simplicity we write $F^M_{\lambda}=F^M_{\textup{t}(\lambda)}$ for
$\lambda\in P^\vee$.

\begin{defi}[Cherednik \cite{CqKZ}]
The $q$-difference equations
\begin{equation}\label{qKZCherB}
F^M_{\lambda}(t)f(q^{-\lambda}t)=f(t),\qquad \lambda\in P^\vee
\end{equation}
for $f\in\mathcal{M}(T)\otimes M$, are called the quantum affine KZ (qKZ)
equations for the $H$-module $M$.
\end{defi}

From the cocycle condition \eqref{qKZcocycleB} and the fact that
$P^\vee$ is an abelian subgroup of $W$, it follows immediately
that the qKZ equations form a holonomic system of
$q$-difference equations, that is,
\[
F^M_\lambda(t)F^M_{\mu}(q^{-\lambda}t)=F^M_\mu(t)F^M_\lambda(q^{-\mu}t)
\]
for all $\lambda,\mu\in P^\vee$.

In this paper we will restrict our attention to a particular
representation of $H$. Recall that $H\simeq H_0\otimes\C_Y[T]$
(cf. Theorem \ref{HcharB}). Fix $\zeta\in T$ and let
$\chi_\zeta\colon\C_Y[T]\to\C$ be the evaluation character
$f(Y)\mapsto f(\zeta)$ for $f\in\C[T]$. We define $M_\zeta$ to be
the induced $H$-module
$M_\zeta:=\textup{Ind}^H_{\C_Y[T]}(\chi_\zeta)=H\otimes_{\chi_\zeta}\C$.
It is the minimal principal series representation of $H$ with
central character $\zeta$. As complex vector spaces we identify
$M_\zeta\simeq H_0$ via
\begin{equation}\label{MzetaIdH0}
T_w\otimes_{\chi_\zeta}1\mapsto T_w, \qquad(w\in
W_0, f\in\C[T]).
\end{equation}
The qKZ equations corresponding to $M_\zeta$ thus can be
viewed as a holonomic system of $q$-difference equations for
meromorphic functions $f(t)$ on $T$ with values in $H_0$. Now
$\mathbb{H}\simeq \C[T]\otimes H$, so that since $H\simeq
H_0\otimes \C[T]$, the  double affine Hecke algebra $\mathbb{H}$
contains another copy of $\C[T]$. In view of Cherednik's duality
anti-isomorphism one might ask, when $\zeta$ is considered as a
variable $\gamma$ on the second torus, whether one can find a set
of $q$-difference equations acting on this central character
$\gamma$, such that together with the original qKZ equations it
makes up a holonomic system of $q$-difference equations for
meromorphic functions $f(t,\gamma)$ on $T\times T$ with values in
$H_0$. The answer turns out to be positive. The idea is as follows.

The construction of the qKZ equations depended on the realization of $W$ inside
the localization of $\mathbb{H}$ by sending the $w$ to the
so-called normalized intertwiners $\rho^{-1}(w)$. Of course, we
can multiply these intertwiners by appropriate factors from
$\C[T]$ to obtain elements $\widetilde{S}_w$ which do live in
$\mathbb{H}$. Clearly, the map $W\to\mathbb{H}^\times$,
$w\mapsto\widetilde{S}_w$  will no longer be a group homomorphism
(like $\rho^{-1}$), but the $\widetilde{S}_w$ still serve as
intertwining elements from which a cocycle can be constructed.
Then Cherednik's duality anti-isomorphism can be invoked to obtain
$Y$-intertwining elements and extend the cocycle to a `double
cocycle' which will give rise to the bispectral quantum KZ
equations. This is explained in the following subsection.

\subsection{Bispectral quantum KZ equations}\label{subsecBqKZB}
The construction of the bispectral quantum KZ equations in the
present setting is more or less the same as in the $\textup{GL}_N$
case, which was done in \cite[\S3]{MS}. Here we repeat the
construction, but, since it is a matter of simply adapting the
notations from \cite{MS}, we omit the proofs.

In view of the last paragraph of the previous subsection we should
first renormalize the intertwiners so that they become members of
$\mathbb{H}$. We put
\[
\begin{split}
\widetilde{S}_i&:=(k_i-k_i^{-1}X^{-a_i^\vee})s_i\in\C(T)\#_qW,\quad i=0,\ldots,N\\
\widetilde{S}_\omega&:=\omega\in\C(T)\#_qW,\quad\omega\in\Omega,
\end{split}
\]
giving rise to the renormalized intertwiners $\widetilde{S}_w$
($w\in W$), defined in the following proposition (see also
\cite[\S1.3]{C}).

\begin{prop}\label{IntPropB}
Let $w=s_{i_1}\cdots s_{i_r}\omega$ be a reduced expression for
$w\in W$ ($i_1,\ldots,i_r\in\{0,\ldots,N\}$, $\omega\in\Omega$).
Then
\begin{itemize}
    \item[{\bf(i)}] $\widetilde{S}_w:=\widetilde{S}_{i_1}\cdots
    \widetilde{S}_{i_r}\widetilde{S}_\omega$ is a
    well-defined element of $\C(T)\# W$;

    \item[{\bf(ii)}] $\widetilde{S}_w\in\mathbb{H}$, in particular
    $\widetilde{S}_i=(1-X^{-a_i^\vee})T_i+(k_i-k_i^{-1})X^{-a_i^\vee}$
    ($0\leq i\leq N$);

    \item[{\bf(iii)}] the $\widetilde{S}_{i}$ ($i=0,\dots,N$)
    satisfy the braid relations (cf. Definition \ref{defAHAB}(i));

    \item[{\bf(iv)}] $\widetilde{S}_wf(X)=(wf)(X)\widetilde{S}_w$
    for $w\in W$, $f\in\C[T]$;

    \item[{\bf(v)}]
    $\widetilde{S}_i\widetilde{S}_i=(k_i-k_i^{-1}X^{a_i^\vee})
    (k_i-k_i^{-1}X^{-a_i^\vee})$ for $i=0,\ldots N$.
\end{itemize}
\end{prop}
For $0\leq i\leq N$ define $d_i\in\C[T]$ by
$d_i(t):=(k_i-k_i^{-1}t^{-a_i^\vee})$. Then for $w\in W$ as in the
proposition we have
\[
\widetilde{S}_w=d_{i_1}(X)(s_{i_1}d_{i_2})(X)\cdots (s_{i_1}\cdots
s_{i_{r-1}}d_{i_r})(X)w.
\]
The proof of part (i) of the proposition relies on the fact that
\[
d_w:=d_{i_1}(s_{i_1}d_{i_2})\cdots (s_{i_1}\cdots
s_{i_{r-1}}d_{i_r})
\]
is independent of the reduced expression for $w$.

Now the `double cocycle' we are going to construct is a cocycle of
$W\times W$. In fact, it turns out to be convenient to anticipate
the role that the anti-involution of $\mathbb{H}$ will play and
extend $W\times W$ as follows. Note that the two-group $\Z_2$ acts
on $W\times W$ by $\iota(w,w')=(w',w)$, where $\iota\in\Z_2$
denotes the nontrivial element. Then we put
\[
\mathbb{W}:=\Z_2\ltimes(W\times W).
\]

Furthermore, the cocycle will act on $H_0$-valued meromorphic
functions on $T\times T$. Let us write
$\mathbb{K}:=\mathcal{M}(T\times T)$ for the field of meromorphic
functions on $T\times T$. Moreover, write
$\mathbb{L}:=\C[T]\otimes\C[T]\simeq \C[T\times T]$ for the ring
of complex valued regular functions on $T\times T$. It acts on
$\mathbb{H}$ via
\begin{equation}\label{LmodstrB}
(f\otimes g)\cdot h:=f(X)hg(Y)
\end{equation}
for $f,g\in\C[T]$ and $h\in \mathbb{H}$. We will usually write
$(t,\gamma)$ for a typical point of $T\times T$. Let
$x=(x_1,\ldots,x_N)$ denote the coordinate functions of the first
copy of $T$ in $T\times T$ and $y=(y_1,\ldots,y_N)$ the coordinate
functions of the second copy. For $f\in\C[T]$ we define
$f(x)\in\mathbb{L}$ by the rule $(t,\gamma)\mapsto f(t)$, and
$f(y)\in\mathbb{L}$ by $(t,\gamma)\mapsto f(\gamma)$. We use the
same conventions for $f(x),f(y)\in\mathbb{K}$ when
$f\in\mathcal{M}(T)$.

An intermediate step in the construction of a $\mathbb{W}$-action
on $H^{\mathbb{K}}=\mathbb{K}\otimes H_0$ are the complex linear
endomorphisms $\sigma_{(w,w')}$ ($w,w'\in W$) of $\mathbb{H}$
defined by
\begin{equation*}
\begin{split}
\sigma_{(w,w^\prime)}(h)&=
\widetilde{S}_wh\widetilde{S}_{w^\prime}^*,\\
\sigma_\iota(h)&=h^*
\end{split}
\end{equation*}
for $h\in\mathbb{H}$. As a corollary of Proposition \ref{IntPropB} we have
\begin{lem}\label{sigmapropB}
The complex linear endomorphisms
$\sigma_{(w,w^\prime)}$ and $\sigma_\iota$
of $\mathbb{H}$ satisfy:
\begin{enumerate}
\item[{\bf (i)}] the $\sigma_{(s_i,e)}$ ($i=0,\ldots,N$)
    satisfy the braid relations;
\item[{\bf (ii)}] $\sigma_{(s_i,e)}^2=d_{s_i}(x)(s_id_{s_i})(x)
    \cdot\textup{id}_{\mathbb{H}}$ for $i=0,\ldots,N$;
\item[{\bf (iii)}] $\sigma_{(\omega,e)}\sigma_{(s_i,e)}
    \sigma_{(\omega^{-1},e)}=\sigma_{(s_{\omega(i)},e)}$
    for $i=0,\ldots,N$ and $\omega\in\Omega$;
\item[{\bf (iv)}] $\sigma_\iota^2=\textup{id}_{\mathbb{H}}$ and
    $\sigma_{(e,w)}=
    \sigma_\iota\sigma_{(w,e)}
    \sigma_{\iota}$ for $w\in W$;
\item[{\bf (v)}]
    $\sigma_{(w,e)}\sigma_{(e,w^\prime)}=
    \sigma_{(w,w^\prime)}=\sigma_{(e,w^\prime)}
    \sigma_{(w,e)}$ for $w,w^\prime\in W$.
\end{enumerate}
\end{lem}

Let us investigate the behavior of these maps under the action of $\mathbb{L}$.
First consider the group involution ${}^\diamond\colon W\rightarrow W$ given by
$w^\diamond=w$ for $w\in W_0$ and
$\lambda^\diamond=-\lambda$ for $\lambda\in P^\vee$. Then
$\mathbb{W}$ acts on $T\times T$ by
\begin{equation*}
\begin{split}
(w,w^\prime)(t,\gamma)&=(wt,w^\prime{}^\diamond\gamma),\\
\iota(t,\gamma)&=(\gamma^{-1},t^{-1})
\end{split}
\end{equation*}
for $w,w^\prime\in W$, where
$t^{-1}:=(t_1^{-1},\ldots,t_N^{-1})\in T$. Transposition yields
an action of $\mathbb{W}$ on $\mathbb{K}$ by field automorphisms and is given by
\begin{equation}\label{doubleactionB}
(\mathrm{w}f)(t,\gamma)=f(\mathrm{w}^{-1}(t,\gamma)), \qquad
\mathrm{w}\in\mathbb{W}.
\end{equation}
Note that $\mathbb{L}=\C[T\times T]$ is a $\mathbb{W}$-subalgebra
of $\mathbb{K}$. As a consequence of the intertwining properties of the
$\widetilde{S}_w$ we have

\begin{lem}\label{actokB}
For $h\in\mathbb{H}$ and $f\in \mathbb{L}$ we have
\begin{equation}\label{catB}
\begin{split}
\sigma_{(w,w^\prime)}(f\cdot h)&=
((w,w^\prime)f)\cdot\sigma_{(w,w^\prime)}(h),\\
\sigma_{\iota}(f\cdot h)&= (\iota f)\cdot
\sigma_{\iota}(h)
\end{split}
\end{equation}
for $w,w^\prime\in W$.
\end{lem}

As $\mathbb{L}$-modules we have
$H_0^{\mathbb{K}}\simeq\mathbb{K}\otimes_{\mathbb{L}}\mathbb{H}$,
so the lemma enables us to extend the maps $\sigma_{(w,w^\prime)}$
($w,w^\prime\in W$) and $\sigma_{\iota}$ to complex linear
endomorphisms of $H_0^{\mathbb{K}}$ for which \eqref{catB} holds
for all $f\in\mathbb{K}$ and $h\in H_0^{\mathbb{K}}$. Note that
the properties of $\sigma_{(w,w^\prime)}$ and $\sigma_\iota$ as
described in Lemma \ref{sigmapropB} also hold true as identities
between endomorphisms of $H_0^{\mathbb{K}}$.

We come to the main result of this subsection. It follows from the
previous observations in the same way as the corresponding result
for $\textup{GL}_N$ (see \cite[Thm. 3.3]{MS}).

\begin{thm}
There is a unique group homomorphism
\[
\tau\colon\mathbb{W}\to\textup{GL}_\C(H_0^{\mathbb{K}})
\]
satisfying
\begin{equation}\label{fishB}
\begin{split}
\tau(w,w^\prime)(f)&=d_w(x)^{-1}d_{w^\prime}^\diamond(y)^{-1}\cdot
\sigma_{(w,w^\prime)}(f),\\
\tau(\iota)(f)&=\sigma_{\iota}(f)
\end{split}
\end{equation}
for $w,w^\prime\in W$ and $f\in H_0^{\mathbb{K}}$. It satisfies
$\tau(\mathrm{w})(g\cdot f)=\mathrm{w}g\cdot \tau(\mathrm{w})(f)$
for $g\in\mathbb{K}$, $f\in H_0^{\mathbb{K}}$ and
$\mathrm{w}\in\mathbb{W}$.
\end{thm}

\begin{rema}\label{remtauxVStau}
Fix $\zeta\in T$. Let $w\in W$ and recall that we write
$\tau_x^{M_\zeta}(w)$ for $\tau_x(w)\in\C(T)\#_q W$ viewed as
endomorphism of $\mathcal{M}(T)\otimes M_\zeta$ as explained in
subsection \ref{subsecqKZB}. Then for $w\in W$,
$f\in\mathcal{M}(T)$ and $h\in H_0\simeq M_\zeta$ (see \eqref{MzetaIdH0}), we have
\[
\tau_x^{M_\zeta}(w)(f\otimes h)=\tau(w,e)(f(x)\otimes
h)(\cdot,\zeta)
\]
as $H_0$-valued meromorphic functions on $T$.
\end{rema}

We are in position to define the $\mathbb{W}$-cocycle with values
in $\textup{GL}_{\mathbb{K}}(H_0^{\mathbb{K}})$, which is a
$\mathbb{W}$-group by the action of $\mathbb{W}$ on the first
tensor leg of $\mathbb{K}\otimes\textup{GL}(H_0)
\simeq\textup{GL}_{\mathbb{K}}(H_0^{\mathbb{K}})$ (cf. subsection
\ref{subsecqKZB}). This $\mathbb{W}$-action on
$\textup{GL}_{\mathbb{K}}(H_0^{\mathbb{K}})$ is denoted without
mentioning the representation map (just as we do for the
$\mathbb{W}$-action on $\mathbb{K}$, cf. \eqref{doubleactionB}).

\begin{cor}\label{CorCocB}
The map $\mathrm{w}\mapsto
C_{\mathrm{w}}:=\tau(\mathrm{w})\mathrm{w}^{-1}$ is a cocycle of
$\mathbb{W}$ with values in the $\mathbb{W}$-group
$\textup{GL}_{\mathbb{K}} (H_0^{\mathbb{K}})$. In other words,
$C_{\mathrm{w}}\in\textup{GL}_{\mathbb{K}}(H_0^{\mathbb{K}})$ and
\[C_{\mathrm{w}\mathrm{w}^\prime}=
C_{\mathrm{w}}\mathrm{w}C_{\mathrm{w}^\prime}\mathrm{w}^{-1}
\]
for all $\mathrm{w},\mathrm{w}^\prime\in\mathbb{W}$.
\end{cor}

In the same way as the cocycle $F_w$ ($w\in W$) in subsection
\ref{subsecqKZB} gave rise to the quantum KZ equations, the
cocycle $C_\mathrm{w}$ ($\mathrm{w}\in\mathbb{W}$) gives rise to a
holonomic system of $q$-difference equations for meromorphic
functions on $T\times T$ with values in $H_0$. By construction we
have
\begin{equation}\label{actionexplB}
(\tau(\mathrm{w})f)(t,\gamma)=
C_\mathrm{w}(t,\gamma)f(\mathrm{w}^{-1}(t,\gamma))
\end{equation}
for $\mathrm{w}\in\mathbb{W}$ and $f\in H_0^{\mathbb{K}}$. For the
sake of simplicity, write
$C_{(\lambda,\mu)}:=C_{(\textup{t}(\lambda),\textup{t}(\mu))}$ for
$\lambda,\mu\in P^\vee$.

\begin{defi}\label{BqKZB}
We call the $q$-difference equations
\begin{equation}\label{BqKZeqnB}
C_{(\lambda,\mu)}(t,\gamma)f(q^{-\lambda}t,q^\mu\gamma)=f(t,\gamma)
\qquad \forall\, \lambda,\mu\in P^\vee,
\end{equation}
the bispectral quantum KZ (BqKZ) equations. We write $\textup{SOL}$
for the set of solutions $f\in H_0^{\mathbb{K}}$ of \eqref{BqKZeqnB}.
\end{defi}

Let $\mathbb{F}\subset\mathbb{K}$ denote the subfield consisting
of $f\in \mathbb{K}$ satisfying
$(\textup{t}(\lambda),\textup{t}(\mu))f=f$ for all $\lambda,\mu\in
P^\vee$. Furthermore let $\mathbb{W}_0$ denote the subgroup
$\Z_2\ltimes (W_0\times W_0)$ of $\mathbb{W}$.

\begin{cor}
{\bf (i)} The BqKZ equations \eqref{BqKZeqnB} form a holonomic
system of $q$-difference equations, that is
\begin{equation}\label{compatibleB}
C_{(\lambda,\mu)}(t,\gamma)C_{(\nu,\xi)}(q^{-\lambda}t,q^{\mu}\gamma)=
C_{(\nu,\xi)}(t,\gamma)C_{(\lambda,\mu)}(q^{-\nu}t,q^{\xi}\gamma)
\end{equation}
for $\lambda,\mu,\nu,\xi\in P^\vee$, as $\textup{End}(H_0)$-valued
meromorphic
functions in $(t,\gamma)\in T\times T$.\\
{\bf (ii)} The solution space $\textup{SOL}$ of BqKZ is a
$\tau(\mathbb{W}_0)$-invariant $\mathbb{F}$-subspace of
$H_0^{\mathbb{K}}$.
\end{cor}

Now fix $\zeta\in T$. By construction, BqKZ (in some sense)
contains Cherednik's qKZ equation associated to the principal
series module $M_\zeta$. Concretely, in view of Remark
\ref{remtauxVStau}, Cherednik's quantum KZ equation
\eqref{qKZCherB} for $M=M_\zeta$ is just
\begin{equation}\label{qKZconcrB}
C_{(\lambda,e)}(t,\zeta)f(q^{-\lambda} t)=f(t),\qquad
\forall\lambda\in P^\vee,
\end{equation}
for $H_0$-valued meromorphic functions $f$ on $T$. In analogy with
BqKZ, we write $\textup{SOL}_\zeta\subset H_0^{\mathcal{M}(T)}$
for the set of solutions of \eqref{qKZconcrB}. Regarding
$H_0^{\mathcal{M}(T)}$ as a vector space over
$\mathcal{E}(T):=\{f\in\mathcal{M}(T)\mid
\textup{t}(\lambda)f=f,\: \forall\lambda\in P^\vee\}$,
$\textup{SOL}_\zeta$ is a $\tau_x^{M_\zeta}(W_0)$-invariant
subspace of $H_0^{\mathcal{M}(T)}$.

\section{Formal principal series representation and the cocycle
values}\label{SectionFormalB}
In this section we investigate the principal series representation
$M_\zeta$ of $H$, when the (fixed) central character $\zeta\in T$
is regarded as a meromorphic variable. This allows us to give
explicit expressions for the cocycle values of the simple
reflections.

\subsection{Formal principal series
representation}\label{subsecFormalB} Recall that
$M_\zeta=\textup{Ind}_{\C_Y[T]}^{H}(\chi_\zeta)$. Now we view
$\C_Y[T]$ as a left $\C_Y[T]$-module by left multiplication and we
put $M:=\textup{Ind}^{H}_{\C_Y[T]}(\C_Y[T])$. We regard $M$ as a
left $H$-module over $\C[T]\simeq\C[\{1\}\times
T]\subset\mathbb{L}$ via
\[
f\cdot (h\otimes_{\C_Y[T]}g(Y))=h\otimes_{\C_Y[T]}(fg)(Y)\qquad
f,g\in\C[T],\,\, h\in H.
\]
Note that $M\simeq\C[\{1\}\times T]\otimes H_0=H_0^{\C[\{1\}\times
T]}$ as modules over $\C[\{1\}\times T]$, hence the representation
map can be regarded as an algebra homomorphism
\[
\eta\colon H\rightarrow \textup{End}_{\C[\{1\}\times
T]}\bigl(H_0^{\C[\{1\}\times T]}\bigr).
\]
Also note that $\textup{End}_{\C[\{1\}\times T]}(H_0^{\C[\{1\}\times
T]}) \simeq \C[\{1\}\times T]\otimes\textup{End}(H_0)$, so we can
and sometimes will regard $\eta(h)$ ($h\in H$) as an
$\textup{End}(H_0)$-valued regular function on $T$ denoted by
$\gamma\mapsto\eta(h)(\gamma)$. By extending the ground ring
$\C[\{1\}\times T]$ to $\mathbb{K}$ we can extend $\eta$ to an
algebra homomorphism
\[\eta\colon H\rightarrow \textup{End}_{\mathbb{K}}(H_0^{\mathbb{K}}).
\]
Similarly, $\eta(h)$ can be viewed as an $\textup{End}(H_0)$-valued function
in $(t,\gamma)\in T\times T$. As such it is constant in $t$, and in case
$h\in H_0$ it is also constant in $\gamma$.

Before being more specific about $\eta$, we need the following
concept (cf. \cite[\S2.6]{M}). A subset $X$ of $P^\vee$ is said to
be saturated if for each $\lambda\in X$ and $\alpha\in R$ we have
$\lambda-r\alpha^\vee\in X$ for all $0\leq
r\leq\langle\lambda,\alpha\rangle$. For $\lambda\in P^\vee$ let
$\Sigma(\lambda)$ denote the smallest saturated subset of $P^\vee$
that contains $\lambda$.

\begin{lem}\label{etaexplicitB}
For $w\in W_0$ and $1\leq i\leq N$ we have
\begin{equation}\label{formulaTIB}
\eta(T_i)T_w=
\begin{cases}
T_{s_iw}\qquad &\hbox{ if } \ell(s_iw)=\ell(w)+1,\\
(k_i-k_i^{-1})T_w+T_{s_iw}\qquad &\hbox{ if }
\ell(s_iw)=\ell(w)-1,
\end{cases}
\end{equation}
and for $p\in\C[T]$ we have
\begin{equation}\label{formulaPYB}
\eta(p(Y))(\gamma)T_e=p(\gamma)T_e
\end{equation}
as regular $H_0$-valued functions in $\gamma$. Moreover, for
$\lambda\in P^\vee$ and $w\in W_0$, we have
\begin{equation}\label{etaYlambdaB}
\eta(Y^\lambda)(\gamma)T_w=\sum_{u\leq
w}p_{u,w}^\lambda(\gamma)T_u,
\end{equation}
where $p_{u,w}^\lambda(\gamma)\in\textup{span}_\C
\{\gamma^\mu\}_{\mu\in\Sigma(\lambda_+)}$ and
$p_{w,w}^\lambda(\gamma)=\gamma^{w^{-1}(\lambda)}$.
\end{lem}
\begin{proof}
Only \eqref{etaYlambdaB} requires proof. We use induction with
respect to the length $\ell(w)$ of $w$, the case $\ell(w)=0$ being
\eqref{formulaPYB}. Next, consider $T_{s_iw}$ with
$\ell(s_iw)=\ell(w)+1$. Using \eqref{eqLusztigB}, we find
\[
\begin{split}
\eta(Y^{\lambda})(\gamma)T_{s_iw}&=\eta(Y^{\lambda}T_i)(\gamma)T_w\\
&=\eta(T_iY^{s_i(\lambda)})(\gamma)T_w+(k_i-k_i^{-1})
\eta\Big(\frac{Y^{\lambda}-Y^{s_i(\lambda)}}
{1-Y^{-\alpha_i^\vee}}\Big)(\gamma)T_w.
\end{split}
\]
Considering the first term we use the induction hypothesis to find
$$\eta(T_iY^{s_i(\lambda)})(\gamma)T_w=T_i\sum_{u\leq w}
\tilde{p}_{u,w}^{s_i(\lambda)}(\gamma)T_u=\sum_{u\leq w}
\tilde{p}_{u,w}^{s_i(\lambda)}(\gamma)T_iT_u,$$ with
$\tilde{p}_{u,w}^{s_i(\lambda)}(\gamma)\in
\textup{span}_\C\{\gamma^\mu\}_{\mu\in\Sigma(s_i(\lambda)_+)}$ and
$\tilde{p}_{w,w}^{s_i(\lambda)}(\gamma)=\gamma^{w^{-1}(s_i(\lambda))}$.
Since $\Sigma(s_i(\lambda)_+)=\Sigma(\lambda_+)$ and
$w^{-1}(s_i(\lambda))=(s_iw)^{-1}(\lambda)$, we can rewrite this
as
\[
\eta(T_iY^{s_i(\lambda)})(\gamma)T_w=\sum_{u\leq s_iw}
p_{u,s_iw}^{\lambda}(\gamma)T_u,
\]
with $p_{u,s_iw}^{\lambda}(\gamma)\in
\textup{span}_\C\{\gamma^\mu\}_{\mu\in\Sigma(\lambda_+)}$ and
$p_{s_iw,s_iw}^{\lambda}(\gamma)=\gamma^{(s_iw)^{-1}(\lambda)}$.

We deal with the second term, the expansion of which will consist
of terms only involving $T_u$ with $u<s_iw$. Set
$n:=\langle\lambda,\alpha_i\rangle$. Note that
\[\frac{Y^{\lambda}-Y^{s_i(\lambda)}}
{1-Y^{-\alpha_i^\vee}}=\left\{\begin{array}{ll}
Y^{\lambda}+Y^{\lambda-\alpha_i^\vee}+\cdots+
Y^{\lambda-(n-1)\alpha_i^\vee},
& n>0,\\
0,& n=0,\\
-Y^{\lambda-n\alpha_i^\vee}-Y^{\lambda-(n+1)\alpha_i^\vee}-\cdots-
Y^{\lambda+\alpha_i^\vee},& n<0,
\end{array}\right.
\]
which is in $\textup{span}_\C\{Y^\mu\}_{\mu\in\Sigma(\lambda_+)}$
in all three cases. We can apply the induction hypothesis to each
of the $Y^\mu$ ($\mu\in\Sigma(\lambda_+)$) to obtain
\[
\eta(Y^\mu)(\gamma)T_w=\sum_{u\leq
w}\check{p}_{u,w}^\mu(\gamma)T_u,
\]
with coefficients $\check{p}_{u,w}^\mu(\gamma)\in\textup{span}_\C
\{\gamma^\nu\}_{\nu\in\Sigma(\mu_+)}$. Since for each
$\mu\in\Sigma(\lambda_+)$ we have $\mu_+\in\Sigma(\lambda_+)$, and
then by \cite[(2.6.3)]{M} $\Sigma(\mu_+)\subset\Sigma(\lambda_+)$,
we obtain the desired expansion.
\end{proof}

We end this subsection by introducing a $\mathbb{K}$-basis of $H_0^{\mathbb{K}}$,
consisting of common eigenfunctions of $\eta(\C_Y[T])$.
Note that $\widetilde{S}^*_w\in H$ for $w\in W_0$. Define
\[
\xi_w:=\eta(\widetilde{S}_{w^{-1}}^*)T_e,\qquad w\in W_0.
\]
Just as we view $\eta(h)$ as $\textup{End}(H_0)$-valued function in different
ways, we will regard $\xi_w$ both as regular $H_0$-valued function in
$\gamma\in T$ and as a meromorphic $H_0$-valued function in
$(t,\gamma)\in T\times T$ (constant in $t$).
\begin{lem}\label{commoneigB}
$\{\xi_w\}_{w\in W_0}$ is a $\mathbb{K}$-basis of
$H_0^{\mathbb{K}}$ consisting of common eigenfunctions for the
$\eta$-action of $\C_Y[T]$ on $H_0^{\mathbb{K}}$. For $p\in\C[T]$
and $w\in W_0$ we have
\begin{equation}\label{eigfcB}
\eta(p(Y))(\gamma)\xi_w(\gamma)=(w^{-1}p)(\gamma)\xi_w(\gamma)
\end{equation}
as $H_0$-valued regular functions in $\gamma\in T$.
\end{lem}

\subsection{The cocycle values}\label{cocyclevaluesB}
Write
\begin{equation}\label{RiB}
R_i(z;\gamma)=c(z;k_i)^{-1}(\eta(T_i)(\gamma)-b(z;k_i)),\qquad 0\leq
i\leq N,
\end{equation}
viewed as a $\textup{End}(H_0)$-valued function which depends
rationally on $z$ and rationally on $\gamma\in T$ for $i=0$ and is
otherwise $\gamma$-independent.
\begin{lem}\label{YBlemB}
{\bf(i)} We have
\begin{equation*}\label{CisRB}
\begin{split}
C_{(s_i,e)}(t,\gamma)&=R_i(t^{a_i^\vee};\gamma),\qquad 0\leq i\leq N,\\
C_{(\omega,e)}(t,\gamma)&=\eta(\omega)(\gamma),\qquad \omega\in\Omega,
\end{split}
\end{equation*}
and $C_\iota$ is the $\mathbb{K}$-linear extension of the anti-algebra involution
of $H_0$ determined by
\[C_{\iota}(T_w)=T_{w^{-1}},\qquad w\in W_0.\]
{\bf(ii)} $R_i(z;\gamma)R_i(z^{-1};\gamma)=\textup{id}$ for $0\leq
i\leq N$.
\end{lem}
\begin{rema}
Note that
\[C_{(e,w)}(t,\gamma)=C_\iota C_{(w,e)}(\gamma^{-1},t^{-1})C_\iota,\qquad
w\in W,\]
so part (i) of the previous lemma uniquely determines $C_{\mathrm{w}}$
for all $\mathrm{w}\in\mathbb{W}$.
\end{rema}

\section{Solutions of BqKZ}\label{SectionSolB}

The main result of this section is the construction of a
particular meromorphic solution $\Phi$ of BqKZ called the basic
asymptotically free solution. The idea is as follows. We first
look for $v\in H_0$ and $G\in\mathbb{K}$ such that $Gv$ will be
the leading term of a solution of BqKZ in some asymptotic region.
These are obtained by looking for a solution of an asymptotic
version of BqKZ, that is, BqKZ in which the $q$-connection
matrices are replaced by their limit values in the asymptotic
region.

Next, we gauge BqKZ by $G$ and look for a power series solution
$\Psi$ of the gauged BqKZ equation converging deep inside the
asymptotic region and which has constant term $v$. By meromorphic
continuation $\Psi$ can be extended to a meromorphic solution of
the gauged BqKZ equation yielding the desired solution
$\Phi=G\Psi\in H_0^\mathbb{K}$ of BqKZ. Apart from the
construction itself we will derive various properties of $\Phi$
and give an explicit $\mathbb{F}$-basis of $\textup{SOL}$, but we
start with the computation of the leading term.

\subsection{The leading term}
In order to find these $v$ and $G$, we first need to compute the
asymptotic leading terms of the $q$-connection matrices
$C_{(\lambda,e)}(t,\gamma)$ ($\lambda\in P^\vee$) as
$|t^{-\alpha^\vee_i}|\rightarrow 0$ ($1\leq i\leq N$).

We define the subring
$\mathcal{A}:=\C[x^{-\alpha^\vee_1},\ldots,x^{-\alpha^\vee_{N}}]$
of $\C[T\times\{1\}]=\C[x_1^{\pm 1},\ldots,x_N^{\pm
1}]\subset\C[T\times T]$. Let $Q(\mathcal{A})$ denote its quotient
field and write $Q_0(\mathcal{A})$ for the subring of
$Q(\mathcal{A})$ consisting of rational functions which are
regular at the point $x^{-\alpha^\vee_i}=0$ ($1\leq i\leq N$). We
consider $Q_0(\mathcal{A})\otimes\C[T]$ as subring of $\C(T\times
T)$ in the natural way.
\begin{lem}\label{asymptoticB}
Let $\lambda\in P^\vee$. We have
\begin{equation}\label{convokB}
C_{(\lambda,e)}\in (Q_0(\mathcal{A})\otimes \C[T])\otimes
\textup{End}(H_0).
\end{equation}
If we write $C_{(\lambda,e)}^{(0)}=
C_{(\lambda,e)}|_{x^{-\alpha^\vee_1}=0,\ldots,x^{-\alpha^\vee_{N}}=0}
\in\C[T]\otimes\textup{End}(H_0)$, we have
\begin{equation}\label{C0lambdaB}
C_{(\lambda,e)}^{(0)}=\delta_{\underline{k}}^\lambda
\eta(T_{w_0}Y^{w_0(\lambda)}T_{w_0}^{-1}).\end{equation}
\end{lem}
\begin{proof}
First we consider $\lambda\in P_+^\vee$. Suppose we have a reduced
expression $\textup{t}(\lambda)=s_{i_1}\cdots s_{i_r}\omega$
($0\leq i_1,\ldots,i_r\leq N$, $\omega\in\Omega$). Then
\begin{equation}\label{ClambdaB}
C_{(\lambda,e)}(t,\gamma)=
R_{i_1}(t^{a^\vee_{i_1}};\gamma)R_{i_2}(t^{s_{i_1}(a^\vee_{i_2})};\gamma)\cdots
R_{i_r}(t^{s_{i_1}\cdots
s_{i_{r-1}}(a^\vee_{i_r})};\gamma)\eta(\omega)(\gamma).
\end{equation}
It follows that $C_{(\lambda,e)}\in(Q(\mathcal{A})\otimes
\C[T])\otimes \textup{End}(H_0)$. Expanding $C_{(-\lambda,e)}$
along the reduced expression
$\textup{t}(-\lambda)=\omega^{-1}s_{i_r}\cdots s_{i_1}$ gives an
expression similar to \eqref{ClambdaB}, from which we conclude
that also $C_{(-\lambda,e)}\in(Q(\mathcal{A})\otimes \C[T])\otimes
\textup{End}(H_0)$. Since the $R_i(z;\gamma)$ are analytic at
$z=0$ and $z=\infty$, we have
$C_{(\lambda,e)},C_{(-\lambda,e)}\in(Q_0(\mathcal{A})\otimes
\C[T])\otimes \textup{End}(H_0)$. Writing an arbitrary weight as
the difference of two dominant weights and using the cocycle
property we conclude \eqref{convokB} for any $\lambda\in P^\vee$.

To prove $\eqref{C0lambdaB}$ we will first compute the limit of
$C_{(\lambda,e)}(t,\gamma)$ as $|t^{\alpha^\vee_i}|\rightarrow 0$
for $1\leq i\leq N$ and then use this together with the cocycle
property to find $C_{(\lambda,e)}^{(0)}(\gamma)$, which is the
limit as $|t^{-\alpha^\vee_i}|\rightarrow 0$ ($1\leq i\leq N$).
Similarly as in the proof of \eqref{convokB}, it suffices to
consider only dominant weights. Assume we have $\lambda\in
P_+^\vee$ a reduced expression for $\textup{t}(\lambda)$ as above
and put $u=s_{i_1}\cdots s_{i_r}$. By formulas (2.2.9) and (2.2.5)
from \cite{M} we have
$\{a_{i_1},s_{i_1}(a_{i_2}),\ldots,s_{i_1}\cdots
s_{i_{r-1}}(a_{i_r})\}=S(u^{-1})=S(\omega^{-1}u^{-1})=S(\textup{t}(-\lambda))$.
Because $\lambda\in P^\vee_+$ we have
\[S(\textup{t}(-\lambda))=\{\alpha+mc\mid\alpha\in R_-,\:1\leq
m\leq -\langle\lambda,\alpha\rangle\}\] (cf. \cite[\S2.4]{M}), and
thus, since $w(a^\vee)=(wa)^\vee$ ($a\in S$, $w\in W$), we have
$|t^{b^\vee}|\to\infty$ ($b\in S(\textup{t}(-\lambda))$) as
$|t^{\alpha^\vee_i}|\rightarrow 0$ ($1\leq i\leq N$). Observe that
$\lim_{z\rightarrow\infty}R_i(z;\gamma)=k^{-1}_i\eta(T_i)(\gamma)$
for $0\leq i\leq N$. It follows that
\[C_{(\lambda,e)}(t,\gamma)\rightarrow k^{-1}_{i_1}\cdots k^{-1}_{i_r}
\eta(Y^{\lambda})(\gamma)=k(\textup{t}(\lambda))^{-1}\eta(Y^{\lambda})(\gamma)
\]
as $|t^{\alpha^\vee_i}|\rightarrow 0$ for all $1\leq i\leq N$.
More generally, we conclude that
\begin{equation}\label{ClambdaTussenB}
C_{(\lambda,e)}(t,\gamma)\rightarrow
\delta_{\underline{k}}^{-\lambda}\eta(Y^{\lambda})(\gamma),\qquad
\lambda\in P^\vee
\end{equation}
as $|t^{\alpha^\vee_i}|\rightarrow 0$ for all $1\leq i\leq N$. In
order to find $C_{(\lambda,e)}^{(0)}$ we use the cocycle property
to write
\[C_{(\lambda,e)}(t,\gamma)=C_{(w_0,e)}(t,\gamma)
C_{(w_0(\lambda),e)}(w_0t,\gamma)C_{(w_0,e)}(q^{-w_0(\lambda)}w_0t,\gamma)
\]
and consider the limit as $|t^{-\alpha^\vee_i}|\rightarrow 0$ for
$1\leq i\leq N$. Note that $C_{(w_0,e)}(t,\gamma)\rightarrow
k(w_0)^{-1}\eta(T_{w_0})$ as $|t^{-\alpha^\vee_i}|\rightarrow 0$
for $1\leq i\leq N$. Hence, using \eqref{ClambdaTussenB},
\[C_{(\lambda,e)}^{(0)}=\delta_{\underline{k}}^{-w_0(\lambda)}\eta(T_{w_0}
Y^{w_0(\lambda)}T_{w_0}^{-1})=\delta_{\underline{k}}^{\lambda}\eta(T_{w_0}
Y^{w_0(\lambda)}T_{w_0}^{-1}),
\]
where the last equality follows from \eqref{klambdaB}.
\end{proof}

The previous lemma implies that the asymptotic form of the quantum
KZ equations
\[C_{(\lambda,e)}(t,\gamma)f(q^{-\lambda}t,\gamma)=f(t,\gamma),\qquad
\lambda\in P^\vee
\]
in the asymptotic region $|t^{\alpha_i^\vee}|\gg0$ ($1\leq i\leq
N$) is
\begin{equation}\label{qKZasymptoticB}
\delta_{\underline{k}}^{\lambda}\eta(T_{w_0}Y^{w_0(\lambda)}T_{w_0}^{-1})
(\gamma)f(q^{-\lambda}t,\gamma)=f(t,\gamma),\qquad \lambda\in
P^\vee.
\end{equation}
Let $\theta_q\in\mathcal{O}(T)$ denote the theta function
associated to the root system $R$ (see \cite{Lo}), defined by
\begin{equation}
\theta_q(t):=\sum_{\lambda\in
P^\vee}q^{\frac{1}{2}\langle\lambda,\lambda\rangle}t^\lambda,
\end{equation}
for $t\in T$. Note that $\theta_q$ is invariant under the action of
$W_0$ on $\mathcal{O}(T)$. Furthermore, it satisfies
$\theta_q(t^{-1})=\theta_q(t)$ and
\begin{equation}\label{thetarelB}
\theta_q(q^\mu t)
=q^{-\frac{1}{2}\langle\mu,\mu\rangle}t^{-\mu}\theta_q(t),
\end{equation}
for all $\mu\in P^\vee$.

Let $G\in\mathbb{K}$ be given by
\begin{equation}\label{WB}
G(t,\gamma):=\frac{\theta_{q}(tw_0(\gamma)^{-1})}{\theta_{q}
(\delta_{\underline{k}}
t)\theta_{q}(\delta_{\underline{k}}^{-1}w_0(\gamma)^{-1})}.
\end{equation}

\begin{prop}\label{GpropB} We have:\\
\textbf{\textup{(i)}} $\iota(G)=G$.\\
\textbf{\textup{(ii)}} $G(t,\gamma)$ satisfies the $q$-difference
equations
\begin{equation}\label{WeqB}
G(q^{-\lambda}t,q^\mu\gamma)=\delta_{\underline{k}}^{-\lambda-\mu}
q^{-\langle w_0(\lambda),\mu\rangle}t^{w_0(\mu)}
\gamma^{-w_0(\lambda)}G(t,\gamma)
\end{equation}
for $\lambda,\mu\in P^\vee$.\\
\textbf{\textup{(iii)}} $f^{(0)}(t,\gamma):=G(t,\gamma)T_{w_0}$ is
a solution of \eqref{qKZasymptoticB} and
$\tau(\iota)f^{(0)}=f^{(0)}$.
\end{prop}
\begin{proof}
By construction we have (i). From \eqref{thetarelB} it follows
that $G$ satisfies $G(q^{-\lambda} t,\gamma)=\delta_{\underline{k}}^{-\lambda}
\gamma^{-w_0(\lambda)}G(t,\gamma)$ for all $\lambda\in P^\vee$. In
view of (i) this suffices to prove (ii). (iii) easily follows from
(i) and (ii).
\end{proof}

\subsection{The basic asymptotically free solution $\Phi$}
As indicated in the introduction of this section we are now going
to gauge BqKZ by $G$. We obtain the gauged $q$-connection matrices
\begin{equation}\label{DlmB}
\begin{split}
D_{(\lambda,\mu)}(t,\gamma)&=G(t,\gamma)^{-1}C_{(\lambda,\mu)}
(t,\gamma)G(q^{-\lambda}t,q^{\mu}\gamma)\\
&=\delta_{\underline{k}}^{-\lambda-\mu}q^{-\langle\mu,w_0(\lambda)\rangle}
\gamma^{-w_0(\lambda)} t^{w_0(\mu)}C_{(\lambda,\mu)}(t,\gamma),
\end{split}
\end{equation}
for $\lambda,\mu\in P^\vee$. It is clear that for $f\in
H_0^\mathbb{K}$ we have $f\in\textup{SOL}$ if and only if
$g:=G^{-1}f\in H_0^\mathbb{K}$ satisfies the holonomic system of
$q$-difference equations
\begin{equation}\label{gaugedeqnB}
D_{(\lambda,\mu)}(t,\gamma)g(q^{-\lambda}t,q^{\mu}\gamma)=g(t,\gamma),\qquad
\lambda,\mu\in P^\vee
\end{equation}
as $H_0$-valued meromorhic functions in $(t,\gamma)\in T\times T$.

We write $\mathcal{B}$ for the analogue of $\mathcal{A}$
corresponding to the second copy of $T$ in $T\times T$. That is,
$\mathcal{B}$ is the subring
$\mathcal{B}:=\C[y^{\alpha^\vee_1},\ldots,y^{\alpha^\vee_{N}}]$ of
$\C[\{1\}\times T]=\C[y_1^{\pm 1},\ldots,y_N^{\pm 1}]$. Similarly,
we write $Q(\mathcal{B})$ for its quotient field and
$Q_0(\mathcal{B})$ for the subring of $Q(\mathcal{B})$ consisting
of rational functions which are regular at the point
$y^{\alpha^\vee_j}=0$ ($1\leq j\leq N$). We consider
$Q_0(\mathcal{A})\otimes\mathcal{B}$ and $\mathcal{A}\otimes
Q_0(\mathcal{B})$ as subrings of $\C(T\times T)$ in the natural
way.

In the proof of the lemma below, we will need a partial order
$\succeq$ on $P^\vee$. First recall the dominance partial order
$\geq$ on $P_+^\vee$, which is defined by
\[
\lambda\geq\mu \Longleftrightarrow \lambda-\mu\in Q_+^\vee,
\]
for $\lambda,\mu\in P_+^\vee$. We can extend this to a partial
order on $P^\vee$ as follows. For $\lambda\in P^\vee$ write
$\lambda_+$ for the unique dominant coweight in the orbit
$W_0\lambda$ and let $\overline{v}_\lambda$ be the shortest $w\in
W_0$ such that $w(\lambda_+)=\lambda$. For $\lambda,\mu\in P^\vee$
we say that $\lambda\succeq\mu$ if either

(i) $\lambda_+>\mu_+$, or

(ii) $\lambda_+=\mu_+$ and $\overline{v}_\lambda\geq\overline{v}_\mu$
(in the Bruhat order).\\
Note that with respect to this order, the anti-dominant coweight
$w_0(\lambda_+)$ is the largest element in the orbit $W_0\lambda$.
More details can be found in \cite[\S2.7]{M}.

The following lemma describes the asymptotic behavior of the
gauged $q$-connection matrices. It allows us to put them in the
context of the general theory of solutions of $q$-difference
equations as described in the appendix of \cite{MS} and is
therefore a key ingredient in the construction of $\Phi$.

\begin{lem}\label{qconnLemB}
Set $A_i=D_{(\varpi^\vee_i,e)}$ and $B_i=D_{(e,\varpi^\vee_i)}$ for
$1\leq
i\leq N$.\\
{\bf (i)} $A_i\in(Q_0(\mathcal{A})\otimes
\mathcal{B})\otimes\textup{End}(H_0)$ and $B_j\in
(\mathcal{A}\otimes Q_0(\mathcal{B}))\otimes \textup{End}(H_0)$.\\
{\bf (ii)} Write $A_i^{(0,0)}\in\textup{End}(H_0)$ and
$B_i^{(0,0)}\in\textup{End}(H_0)$ for the value of $A_i$ and $B_i$
at $x^{-\alpha^\vee_r}=0=y^{\alpha^\vee_s}$ ($1\leq r,s\leq N$).
For $w\in W_0$ we have
\begin{equation}\label{AasB}
A_i^{(0,0)}(T_{w_0}T_w)=
\begin{cases}
0\quad &\hbox{ if }\,\, w^{-1}w_0(\varpi^\vee_i)\not=w_0(\varpi^\vee_i),\\
T_{w_0}T_w\quad &\hbox{ if }\,\,
w^{-1}w_0(\varpi^\vee_i)=w_0(\varpi^\vee_i)
\end{cases}
\end{equation}
and
\begin{equation}\label{BasB}
B_i^{(0,0)}(T_{w_0}T_w)=
\begin{cases}
0\quad &\hbox{ if }\,\, w(\varpi^\vee_i)\not= \varpi^\vee_i,\\
T_{w_0}T_w\quad &\hbox{ if }\,\, w(\varpi^\vee_i)=\varpi^\vee_i.
\end{cases}
\end{equation}
\end{lem}
\begin{proof}
We give the proof of (i), which differs substantially from the
$\textup{GL}_N$ case (cf. \cite[Lem. 5.2]{MS}), and omit the proof
of (ii) which is similar. By \eqref{DlmB} we have
\[
A_i(t,\gamma)=
\delta_{\underline{k}}^{-\varpi^\vee_i}\gamma^{-w_0(\varpi^\vee_i)}
C_{(\varpi^\vee_i,e)}(t,\gamma).
\]
Because of \eqref{convokB} we only need to worry about the
$\gamma$-dependence of $A_i(t,\gamma)$.

Let $\textup{t}(\varpi_i^\vee)=\omega s_{i_1}\cdots s_{i_r}$
($\omega\in\Omega$, $0\leq i_1,\ldots,i_r\leq N$) be a reduced
expression. Then, in view of the cocycle condition, Lemma
\ref{YBlemB} and formula \eqref{RiB},
\[
C_{(\varpi_i^\vee,e)}(t,\gamma)=\eta(\omega)(\gamma)C_{(s_{i_1}\cdots
s_{i_r},e)}(\omega^{-1}t,\gamma)= \sum_{w\leq
t(\varpi_i^\vee)}a_w(t)\eta(T_w)(\gamma)
\]
for certain $a_w\in Q_0(\mathcal{A})$. Now consider such $w\in W$
with $w\leq \textup{t}(\varpi_i^\vee)$. We have a unique
decomposition $w=\textup{t}(\lambda)\widetilde{w}$, with
$\lambda=w(0)\in P^\vee$ and $\widetilde{w}\in W_0$. Then
\[
\textup{t}(\lambda)=\textup{t}(\overline{v}_\lambda(\lambda_+))=
\overline{v}_\lambda
\textup{t}(\lambda_+)\overline{v}_\lambda^{-1},
\]
hence $w=\overline{v}_\lambda \textup{t}(\lambda_+)
\overline{v}_\lambda^{-1}\widetilde{w}$. Multiple use of
\cite[(3.1.7)]{M} yields
$T_w=hT_{\textup{t}(\lambda_+)}h'=hY^{\lambda_+}h'$ for some
$h,h'\in H_0$, hence
\[
\eta(T_w)(\gamma)=\eta(h)\eta(Y^{\lambda_+})(\gamma)\eta(h').
\]
It remains to show that
$\gamma^{-w_0(\varpi_i^\vee)}\eta(Y^{\lambda_+})(\gamma)\in\mathcal{B}\otimes
\textup{End}(H_0)$. We can use \eqref{etaYlambdaB} to write
\[
\eta(Y^{\lambda_+})(\gamma)T_w=\sum_{u\leq
w}p_{u,w}^{\lambda_+}(\gamma)T_u
\]
with $p^{\lambda_+}_{u,w}(\gamma)\in
\textup{span}_\C\{\gamma^{\mu}\}_{\mu\in\Sigma(\lambda_+)}$ and
$p_{w,w}^{\lambda_+}(\gamma)=\gamma^{w^{-1}(\lambda_+)}$. Thus we
need to show that
\[
\gamma^{-w_0(\varpi_+)+\mu}\in\mathcal{B}\qquad\forall\mu\in\Sigma(\lambda_+),
\]
i.e., that $-w_0(\varpi^\vee_i)+\mu\in Q_+^\vee$ for all
$\mu\in\Sigma(\lambda_+)$. Since $\Sigma(\lambda_+)$ is
$W_0$-invariant and $w_0(Q_+^\vee)=-Q_+^\vee$, this is equivalent
to showing that $-\varpi^\vee_i+\mu\in-Q_+^\vee$ for all
$\mu\in\Sigma(\lambda_+)$, or
\[
\varpi_i^\vee-\mu\in Q_+^\vee\qquad\forall\mu\in\Sigma(\lambda_+).
\]
Now the fact that $w\leq \textup{t}(\varpi_i^\vee)$ in the Bruhat
order on $W$, implies that $\lambda\preceq\varpi_i^\vee$ (cf.
\cite[(2.7.11)]{M}), and hence either $\lambda_+=\varpi_i^\vee$ or
$\lambda_+<\varpi_i^\vee$. Fix $\mu\in\Sigma(\lambda_+)$. In the
first case, if $\lambda_+=\varpi_i^\vee$, we have
$\mu\in\varpi_i^\vee-Q_+^\vee$, since
\[
\Sigma(\varpi_i^\vee)=\bigcap_{v\in W_0}v(\varpi_i^\vee-Q_+^\vee)
\]
by \cite[(2.6.2)]{M}. Hence $\varpi_i^\vee-\mu\in Q_+^\vee$. In the
second case, if $\lambda_+<\varpi_i^\vee$, then
$\Sigma(\lambda_+)\subset\Sigma(\varpi_i^\vee)$ by
\cite[(2.6.3)]{M}, and again $\mu\in\varpi_i^\vee-Q_+^\vee$. This
concludes the proof for $A_i$. For $B_i$, use that
$C_{(e,\varpi_i^\vee)}(t,\gamma)=C_\iota
C_{(\varpi_i^\vee,e)}(\gamma^{-1},t^{-1})C_\iota$.
\end{proof}

Part (ii) of the previous lemma asserts that the endomorphisms
$A_i^{(0,0)}$ and $B_i^{(0,0)}$ are semisimple. Similarly as for
$\textup{GL}_N$, the main theorem follows from the lemma together
with the general theory of solutions of $q$-difference equations
as described in the appendix of \cite{MS} (in particular \cite[Thm
8.6]{MS}).

For $\epsilon>0$, put $B_{\epsilon}:=\{t\in T\mid
|t^{\alpha^\vee_i}|<\epsilon\mbox{ for }1\leq i\leq N\}$ and
$B_{\epsilon}^{-1}:=\{t\in T\mid t^{-1}\in B_\epsilon\}$.

\begin{thm}\label{asymTHMB}
There exists a unique solution $\Psi\in H_0^{\mathbb{K}}$ of the
gauged equations \eqref{gaugedeqnB} such that, for some
$\epsilon>0$,\\
{\bf (i)} $\Psi(t,\gamma)$ admits an $H_0$-valued power series
expansion
    \begin{equation}\label{psB}
        \Psi(t,\gamma)=\sum_{\alpha,\beta\in
        Q^\vee_+}K_{\alpha,\beta}t^{-\alpha}\gamma^{\beta},
        \qquad(K_{\alpha,\beta}\in H_0)
    \end{equation}
for $(t,\gamma)\in B_\epsilon^{-1}\times B_\epsilon$ which is
normally convergent on compacta of $B_\epsilon^{-1}\times
B_\epsilon$. In particular, $\Psi(t,\gamma)$ is analytic at
$(t,\gamma)\in B_\epsilon^{-1} \times B_\epsilon$;\\
{\bf(ii)} $K_{0,0}=T_{w_0}$.
\end{thm}
\begin{proof}
We only remark that in order to match the present situation with
the one considered in \cite[\S8]{MS}, one should take in
\cite[\S8]{MS}: $M=2N$, $A_i=A_i^{(0,0)}$, $A_{N+i}=B_i^{(0,0)}$
and $q_i=q^{2/\|\alpha_i\|^2}$ for $1\leq i\leq N$ and variables
$z_i=x^{-\alpha^\vee_i}$ and $z_{N+j}=y^{\alpha^\vee_j}$ for
$1\leq i,j\leq N$.
\end{proof}

\begin{defi}
We call $\Phi:=G\Psi\in\textup{SOL}$ the basic asymptotically free
solution of BqKZ.
\end{defi}

The $\tau(\iota)$-invariance of $\textup{SOL}$, the
$\iota$-invariance of $G$, and the uniqueness part of Theorem
\ref{asymTHMB} imply that $\Phi$ enjoys the following duality
property.
\begin{thm}[Duality]\label{selfdualTHMB}
The basic asymptotically free solution $\Phi$ of BqKZ is self-dual,
in the sense that
\[
\tau(\iota)\Phi=\Phi.
\]
\end{thm}

\subsection{Singularities}
In this subsection we have a closer look at the analytic
properties of $\Psi$. Write $q_\alpha:=q^{2/\|\alpha\|^2}$ for
$\alpha\in R$ and set
\[\mathcal{S}_+:=\{t\in T \,\, | \,\, t^{\alpha^\vee}\in k_\alpha^{-2}
q_\alpha^{-\mathbb{N}}\, \textup{ for some }\, \alpha\in R_+\}.
\]
\begin{prop}\label{PsiAnalyticB}
The $H_0$-valued meromorphic function $\Psi$ is analytic on
$T\setminus\mathcal{S}_+^{-1}\times T\setminus\mathcal{S}_+$.
\end{prop}
\begin{proof}
Let $\lambda,\mu\in P_+^\vee$. By \eqref{DlmB} and the cocycle
property, $D_{(\lambda,\mu)}(t,\gamma)$ is regular at
$(t,\gamma)=(s,\zeta)$ if
$C_{(\varpi^\vee_i,\varpi^\vee_j)}(q^{-\nu}t,q^\xi\gamma)$ is
regular at $(t,\gamma)=(s,\zeta)$ for all $1\leq i,j\leq N$ and
$\xi,\nu\in P^\vee_+$. This in turn holds, again by virtue of the
cocycle property together with \eqref{convokB}, if
$C_{(\omega_j^\vee,e)}(q^{-\nu}t,\gamma)$ is regular at
$(t,\gamma)=(s,\zeta)$ for all $\nu\in P^\vee_+$ and $1\leq j\leq
N$.

Suppose we have a reduced expression
$\textup{t}(\varpi^\vee_j)=s_{i_1}\cdots s_{i_r}\omega$ ($1\leq
j\leq N$). Similarly as in the proof of Lemma \ref{asymptoticB}, we have
\begin{equation*}\label{fundamentalB}
\begin{split}
C_{(\varpi^\vee_j,e)}(t,\gamma)=
R_{i_1}(t^{a^\vee_{i_1}};\gamma) \cdots
R_{i_r}(t^{s_{i_1}\cdots
s_{i_{r-1}}(a^\vee_{i_r})};\gamma)\eta(\omega)(\gamma),
\end{split}
\end{equation*}
and
\[
\begin{split}
\{a_{i_1},s_{i_1}(a_{i_2}),\ldots,s_{i_1}\cdots
s_{i_{r-1}}(a_{i_r})\}&=S(\textup{t}(-\varpi_j^\vee))\\
&=\{\alpha+mc\mid\alpha\in R_-,\:1\leq
m\leq-\langle\varpi_j^\vee,\alpha\rangle\}
\end{split}
\]
Now $R_i(z;\gamma)$ has only a simple pole at $z=k_i^{-2}$, so
$C_{(\varpi^\vee_j,e)}(t,\gamma)$ has possibly poles at
\[
t^{a^\vee}=k_a^{-2},\qquad a\in S(\textup{t}(-\varpi_j^\vee)).
\]
Note that
\[t^{(\alpha+mc)^\vee}=t^{\alpha^\vee+(2m/\|\alpha\|^2)c}=
q_\alpha^{m}t^{\alpha^\vee},\] hence there are possibly poles at
\[
q_\alpha^{m}t^{\alpha^\vee}=k_\alpha^{-2},\qquad\alpha\in R_-,\:
1\leq m\leq-\langle\varpi_j^\vee,\alpha\rangle,
\]
or, equivalently, at
\[
t^{-\alpha^\vee}=q_\alpha^{-m}k_\alpha^{-2},\qquad\alpha\in R_+,\:
1\leq m\leq\langle\varpi_j^\vee,\alpha\rangle.
\]
Consequently, $C_{(\varpi^\vee_j,e)}(q^{-\nu}t,\gamma)$ is regular at
$t\in T\setminus\mathcal{S}^{-1}_+$ for all $\nu\in P_+^\vee$. By
the considerations in the previous paragraph we conclude that
$D_{(\lambda,\mu)}(t,\gamma)$ is regular at $(t,\gamma)\in
T\setminus\mathcal{S}_+^{-1}\times T\setminus\mathcal{S}_+$ for
all $\lambda,\mu\in P_+^\vee$.

Let $U\times V$ be a relatively compact open subset of
$T\setminus\mathcal{S}_+^{-1}\times T\setminus\mathcal{S}_+$.
Choose $\lambda,\mu\in P_+^\vee$ such that the closure of
$q^{-\lambda}U\times q^{\mu}V$ is contained in
$B_\epsilon^{-1}\times B_\epsilon$. Then as meromorphic
$H_0$-valued function in $(t,\gamma)\in U\times V$, we have
\begin{equation}\label{PsiDPsiB}
\Psi(t,\gamma)=D_{(\lambda,\mu)}(t,\gamma)\Psi(q^{-\lambda}t,q^\mu\gamma),
\end{equation}
and by Theorem \ref{asymTHMB}(i) the proof is now complete.
\end{proof}
\begin{rema}
The previous proposition gives in particular information about the singularities of the basic asymptotic solution $\Phi=G\Psi$. Unfortunately, it is not possible to precisely pinpoint the singularities of $G$. To overcome this issue we could choose a different theta function in the definition of $G$, namely one for which we have a product formula available. The price we pay is that we have to enlarge the torus $T$. Let $\vartheta_q\in\mathcal{M}(T)$ denote the renormalized Jacobi
theta function
\begin{equation}\label{thetaB}
\vartheta_q(z):=\prod_{m\geq0}(1-q^mz)(1-q^{m+1}/z)
\end{equation}
for $z\in \C^\times$. It satisfies
\begin{equation}\label{thetafuncB}
\vartheta_q(q^mz)=(-z)^{-m}q^{-\frac{1}{2}m(m-1)}\vartheta_q(z),\qquad
m\in\Z.
\end{equation}
Let $e\in\mathbb{N}$ be the unique positive
integer such that $e\langle P^\vee,P^\vee\rangle=\Z$. For all $a\in S$, fix $k_a^{1/6e}$ such that $k_a^{1/6e}=k_{w(a)}^{1/6e}$ for all $w\in W$.
Now put $T':=\textup{Hom}_\Z(6e P^\vee,\C^\times)$.
The canonical map
$T'\twoheadrightarrow T$ gives rise to
an embedding $\mathcal{M}(T\times
T)\hookrightarrow\mathcal{M}(T'\times T')$. Now define $\widehat{G}\in\mathcal{M}(T'\times T')$ by
\begin{equation}\label{GB}
\widehat{G}(t,\gamma):=\prod_{i,j=1}^N
\left(\frac{\vartheta_{q^{1/e}}(\kappa_j^{-1/e}
t^{\alpha_i/e})\vartheta_{q^{1/e}}(\kappa_{i}^{-1/e}
\gamma^{w_0(\alpha_{j})/e})}
{\vartheta_{q^{1/e}}(t^{\alpha_i/e}\:\gamma^{w_0(\alpha_j)/e})}
\right)^{e\langle\varpi_i^\vee,\varpi_j^\vee\rangle},
\end{equation}
where $\kappa_j^{1/e}:=\prod_{\alpha\in
R_+}k_\alpha^{\langle\alpha_j,\alpha\rangle/e}$. Then $\widehat{G}$
satisfies the properties of Proposition \ref{GpropB}.
\end{rema}

\begin{cor}\label{PsiGammaExpansionB}
{\bf (i)} Write $\Psi(t,\gamma)=\sum_{\alpha\in
Q_+^\vee}\Gamma_\alpha(\gamma)t^{-\alpha}$ for $(t,\gamma)\in
B_{\epsilon}^{-1}\times B_\epsilon$, with $\Gamma_\alpha$
($\alpha\in Q_+^\vee$) the analytic $H_0$-valued function
$\Gamma_\alpha(\gamma):=\sum_{\beta\in
Q_+^\vee}K_{\alpha,\beta}\gamma^\beta$ on $B_\epsilon$. Then each
$\Gamma_\alpha$ can uniquely be extended to a meromorphic
$H_0$-valued function on $T$, analytic on
$T\setminus\mathcal{S}_+$, such that for $(t,\gamma)\in
B_{\epsilon}^{-1}\times T\setminus\mathcal{S}_+$
\[
\Psi(t,\gamma)=\sum_{\alpha\in
Q_+^\vee}\Gamma_\alpha(\gamma)t^{-\alpha},
\]
converging normally on compacta of $B_{\epsilon}^{-1}\times
T\setminus\mathcal{S}_+$.\\
{\bf (ii)} The leading term $\Gamma_0$
satisfies
\begin{equation}\label{gamma0B}
\Gamma_0(\gamma)=K(\gamma)T_{w_0},
\end{equation}
for some $K\in\mathcal{M}(T)$.
\end{cor}
\begin{proof}
(i) See \cite[Lemma 5.7]{MS}.\\
(ii) This is also similar as in \cite{MS}, but for the convenience
of the reader we provide the details. $\Psi$ satisfies
$A_i(t,\gamma)\Psi(q^{-\varpi_i^\vee}t,\gamma)=\Psi(t,\gamma)$ for
$1\leq i\leq N$. Considering the limit $|t^{-\alpha_j^\vee}|\to0$,
we obtain
\[
\gamma^{-w_0(\varpi_i^\vee)}\eta(T_{w_0}Y^{w_0(\varpi_i^\vee)}T_{w_0}^{-1})
(\gamma)\Gamma_0(\gamma)=\Gamma_0(\gamma)
\]
for $1\leq i\leq N$, and in view of Lemma \ref{commoneigB} this
forces
\[
\Gamma_0(\gamma)=K(\gamma)\eta(T_{w_0})\xi_e(\gamma)=K(\gamma)T_{w_0}
\]
for some $K\in\mathcal{M}(T)$.
\end{proof}

\begin{rema}
In the following section we will give an explicit formula for
$K(\gamma)$. It will follow immediately from an explicit formula
for the leading term of the so-called Harish-Chandra series
solution of a bispectral problem corresponding to BqKZ. In
\cite{MS}, for $\textup{GL}_N$, it is exactly the way around.
There, the latter is found as a consequence of an explicit formula
for $K(\gamma)$, which in turn is due to rather explicit
expressions for the $q$-connection matrices of BqKZ.
\end{rema}

From Proposition \ref{PsiAnalyticB} and its corollary we obtain
the following result for specialized spectral parameter.
\begin{cor}\label{corPsiSpecB}
Fix $\zeta\in T\setminus\mathcal{S}_+$. \\
{\bf (i)} The $H_0$-valued meromorphic
function $\Psi(t,\gamma)$ in $(t,\gamma)\in T\times T$ can be specialized
at $\gamma=\zeta$, giving rise to a meromorphic $H_0$-valued function
$\Psi(t,\zeta)$ in $t\in T$, which is regular at
$t\in T\setminus\mathcal{S}_+^{-1}$.\\
{\bf (ii)} For $t\in B_\epsilon^{-1}$ we have the power series expansion
\[
\Psi(t,\zeta)=\sum_{\alpha\in Q^\vee_+}\Gamma_\alpha(\zeta)t^{-\alpha},
\]
converging normally on compacta of $B_\epsilon^{-1}$.\\
{\bf (iii)} $\Psi(t,\zeta)$ satisfies the system of $q$-difference
equations
\begin{equation}\label{specialgaugeB}
D_{(\lambda,e)}(t,\zeta)\Psi(q^{-\lambda}t,\zeta)=\Psi(t,\zeta),\qquad
\forall \lambda\in P^\vee.
\end{equation}
\end{cor}

\subsection{Consistency} BqKZ is a holonomic system of first-order
$q$-difference equations with connection matrices depending
rationally on $(t,\gamma)\in T\times T$ and therefore it is
consistent (see \cite[Prop. 5.2]{Et}). This means that
$\textup{dim}_{\mathbb{F}}(\textup{SOL})=\textup{dim}_{\C}(H_0)$,
or, equivalently, that BqKZ allows a so-called fundamental matrix
solution $U$. In \cite{Et}, such a fundamental matrix solution was
found by algebraic geometric arguments. A different approach,
using the asymptotic solution $\Phi$, was taken in \cite{MS}. Here
we shortly repeat this latter approach for arbitrary root systems.
The advantage of this approach is that it produces a basis of
$\textup{SOL}$ in terms of asymptotically free solutions. For
details we refer to \cite[\S5.6]{MS}.

We say that
$F\in\textup{End}(H_0)^\mathbb{K}=\mathbb{K}\otimes\textup{End}(H_0)$
is an $\textup{End}(H_0)$-valued solution of BqKZ, if
\[C_{(\lambda,\mu)}(t,\gamma)F(q^{-\lambda}t,q^{\mu}\gamma)=
F(t,\gamma),\qquad \lambda,\mu\in P^\vee,
\]
as $\textup{End}(H_0)$-valued meromorphic functions in
$(t,\gamma)\in T\times T$.

Define $U\in\textup{End}(H_0)^{\mathbb{K}}$ by
\begin{equation}
U\bigl(k(w)^{-1}T_{w_0}T_{w^{-1}}\bigr):=\tau(e,w)\Phi
\end{equation}
for $w\in W_0$.
\begin{prop}\label{basispropB}
We have\\
{\bf (i)} $U\in\textup{End}(H_0)^\mathbb{K}$ is an invertible solution of BqKZ
with values in $\textup{End}(H_0)$. In particular, identifying
$\textup{End}(H_0)^\mathbb{K}\simeq\textup{End}_\mathbb{K}(H_0^\mathbb{K})$
as $\mathbb{K}$-algebras, we have $U\in\textup{GL}_\mathbb{K}(H_0^\mathbb{K})$.\\
{\bf (ii)} $U^\prime\in\textup{End}(H_0)^{\mathbb{K}}$ is an
$\textup{End}(H_0)$-valued meromorphic solution of BqKZ if and
only if
$U^\prime=UF$ for some $F\in\textup{End}(H_0)^{\mathbb{F}}$.\\
{\bf (iii)} $U$, viewed as $\mathbb{K}$-linear endomorphism of
$H_0^{\mathbb{K}}$, restricts to an $\mathbb{F}$-linear isomorphism
$U\colon H_0^{\mathbb{F}}\rightarrow \textup{SOL}$.\\
{\bf (iv)} $\{\tau(e,w)\Phi\}_{w\in W_0}$ is an
$\mathbb{F}$-basis of $\textup{SOL}$.
\end{prop}

\begin{rema}
The quantum KZ equations \eqref{qKZconcrB} form a consistent
system of $q$-difference equations as well. For generic $\zeta\in
T$ (that is, for $\zeta\in T$ where $\Phi(t,\gamma)$ can be
specialized in $\gamma=\zeta$ and moreover $U(\cdot,\zeta)$ is
invertible), this follows along the same line as above, but of
course one can use \cite[Prop. 5.2]{Et} again, which applies for
all $\zeta\in T$.
\end{rema}

\section{Correspondence with bispectral
problems}\label{SectionCorrespB}

For the principal series representation $M_\zeta$ ($\zeta$
generic) of $H$, Cherednik \cite[Thm. 3.4]{Cref} constructed a map
which embeds the associated solution space of the quantum affine
KZ equation \eqref{qKZCherB} into the solution space of a system
of $q$-difference equations, involving the Macdonald
$q$-difference operator. This is a special case of a
correspondence between the solutions of the quantum affine KZ
equations associated with an arbitrary finite-dimensional
$H$-module $M$ and a more general system of $q$-difference
equations (see \cite{CInd}).

We will consider the map when $M$ is the formal principal series
module $M=\textup{Ind}^{H}_{\C_Y[T]}(\C_Y[T])$ (see Subsection
\ref{subsecFormalB}). In this case Cherednik's correspondence
yields an embedding $\chi_+$ of $\textup{SOL}$ into the solution
space of a bispectral problem for the Macdonald $q$-difference
operators.

\subsection{The bispectral problem for the Macdonald $q$-difference operators}
Using the action of $\mathbb{W}$ on $\C(T\times T)$ given by
\eqref{doubleactionB}, we can form the smash product algebra
$\C(T\times T)\#\mathbb{W}$. It contains $\C(T)\#_q W\simeq
\C(T\times\{1\})\#(W\times\{e\})$ and $\C(T)\#_{q^{-1}}W\simeq
\C(\{1\}\times T)\#(\{e\}\times W)$ as subalgebras. In this
interpretation, Cherednik's algebra homomorphism
$\rho_{\underline{k}^{-1},q}\colon H(k^{-1})\to\C(T)\#_q W$ (see
Theorem \ref{CherednikB}) gives rise to an algebra homomorphism
\[
\rho^{x}_{\underline{k}^{-1},q}\colon
H(\underline{k}^{-1})\to\C(T\times T)\#\mathbb{W},
\]
considered as $q$-difference reflection operators in the first
torus variable, and similarly $\rho_{\underline{k},q^{-1}}\colon
H(\underline{k})\to\C(T)\#_{q^{-1}}W$ to an algebra homomorphism
\[
\rho^{y}_{\underline{k},q^{-1}}\colon
H(\underline{k})\to\C(T\times T)\#\mathbb{W},
\]
considered as $q$-difference reflection operators in the second
torus variable. Note that the images of
$\rho^{x}_{\underline{k}^{-1},q}$ and
$\rho^{y}_{\underline{k},q^{-1}}$ in $\C(T\times T)\#\mathbb{W}$
commute, so we can form the algebra homomorphism
\[
\rho^{x}_{\underline{k}^{-1},q}\otimes\rho^{y}_{\underline{k},q^{-1}}\colon
H(\underline{k}^{-1})\otimes H(\underline{k})\to\C(T\times
T)\#\mathbb{W}.
\]
The maps $\rho^{x}_{\underline{k}^{-1},q}$ and
$\rho^{y}_{\underline{k},q^{-1}}$ are related as follows.
\begin{lem}\label{lemrhodualityB}
Let ${}^\circ\colon H(\underline{k}^{-1})\to H(\underline{k})$ be
defined as the unique algebra isomorphism satisfying
\[
T_i^\circ= T_i^{-1},\quad \omega^\circ=\omega,
\]
for $0\leq i\leq N$ and $\omega\in\Omega$. Then we have
\begin{equation}\label{rhodualityB}
\rho^{y}_{\underline{k},q^{-1}}(h^\circ)=
\iota\rho^{x}_{\underline{k}^{-1},q}(h)\iota
\end{equation}
for all $h\in H(\underline{k}^{-1})$.
\end{lem}
\begin{proof}
Since $\rho^{x}_{\underline{k}^{-1},q}$,
$\rho^{y}_{\underline{k},q^{-1}}$ and ${}^\circ$ are algebra
homomorphisms, the lemma follows by verifying \eqref{rhodualityB}
for $T_i$ ($0\leq i\leq N$) and $\omega\in\Omega$. Let $0\leq
i\leq N$ and $f\in\mathbb{K}$. In $H(\underline{k})$, we have
$T^{-1}_i=T_i+k_i^{-1}-k_i$, hence, on the one hand,
\[
\big(\rho^y_{\underline{k},q^{-1}}(T_i^{-1})f\big)(t,\gamma)=
k_i^{-1}f(t,\gamma)+
c_{a_i;\underline{k},q^{-1}}(\gamma)\big(f(t,s_i^\diamond\gamma)-f(t,\gamma)
\big).
\]
On the other hand,
\[
\begin{split}
\big(\iota\rho^x_{\underline{k}^{-1},q}&(T_i)\iota f\big)(t,\gamma)=
\big(\rho^x_{\underline{k}^{-1},q}(T_i)\iota f\big)(\gamma^{-1},t^{-1})\\
&=k_i^{-1}(\iota f) (\gamma^{-1},t^{-1})+
c_{a_i;\underline{k}^{-1},q}(\gamma^{-1})\big((\iota f)
(s_i\gamma^{-1},t^{-1})-(\iota f)(\gamma^{-1},t^{-1})\big)\\
&=k_i^{-1}f(t,\gamma)+c_{a_i;\underline{k},q^{-1}}(\gamma)
\big( f(t,s^\diamond_i\gamma)-f(t,\gamma)\big),
\end{split}
\]
where we used \eqref{cParamRelB} for the last equality. The
verification for $\omega\in\Omega$ is easier and left to the
reader.
\end{proof}

By means of the canonical action of $\C(T\times T)\#\mathbb{W}$ on
$\C(T\times T)$, the subalgebra $\mathbb{D}:=\C(T\times T)\#
(P^\vee\times P^\vee)\subset\C(T\times T)\#\mathbb{W}$ can be
identified with the algebra of $q$-difference operators on
$T\times T$ with rational coefficients. Any element
$D\in\C(T\times T)\#\mathbb{W}$ has an expansion
\begin{equation}\label{DexpansionB}
D=\sum_{\mathrm{w}\in\mathbb{W}_0}D_{\mathrm{w}}\mathrm{w},
\end{equation}
with $D_\mathrm{w}\in\mathbb{D}$. Since this expansion is unique,
we have a well-defined $\C(T\times T)$-linear map
$\textup{Res}\colon\C(T\times T)\#\mathbb{W}\to\mathbb{D}$,
determined by
\[
\textup{Res}(D):=\sum_{\mathrm{w}\in\mathbb{W}_0}D_{\mathrm{w}},
\]
with $D\in\C(T\times T)\#\mathbb{W}$ given as in
\eqref{DexpansionB}. Let $\C(T\times T)^{\mathbb{W}_0}$ denote the
field of $\mathbb{W}_0$-invariant rational functions on $T\times
T$. Restricted to $\C(T\times T)^{\mathbb{W}_0}$, we have
$D|_{\C(T\times T)^{\mathbb{W}_0}}=\textup{Res}(D)|_{\C(T\times
T)^{\mathbb{W}_0}}$ for all $D\in\C(T\times T)\#\mathbb{W}$.

It is well-known (see, e.g., \cite[(4.2.10)]{M}) that the center
$Z(H)$ of the affine Hecke algebra $H$ is $\C_Y[T]^{W_0}$. For
$p\in\C[T]^{W_0}$, set
\[
L_p^x:=\textup{Res}(\rho^{x}_{\underline{k}^{-1},q}(p(Y)))\in\mathbb{D},
\]
where $p(Y)$ is considered as element of $Z(H(\underline{k}^{-1}))$, and set
\[
L_p^y:=\textup{Res}(\rho^{y}_{\underline{k},q^{-1}}(p(Y)))\in\mathbb{D},
\]
where $p(Y)$ is considered as element of $Z(H(\underline{k}))$. It
is well-known that the operators $L_p^x$ (and hence $L_p^y$) are
pairwise commuting and $(W_0\times W_0)$-invariant, and by
construction $[L_p^x,L_{p'}^y]=0$ in $\mathbb{D}$ for all
$p,p'\in\C[T]^{W_0}$. The operators $L_p^x$ and $L_p^y$ are
related as follows.
\begin{lem}
For $p\in\C[T]^{W_0}$, we have
\begin{equation}\label{LdualityB}
L_p^y=\iota L_p^x\iota.
\end{equation}
\end{lem}
\begin{proof}
Similarly as for $\textup{GL}_N$ (see \cite[\S6.2]{MS}), the lemma
follows from \eqref{rhodualityB} together with the fact that
\begin{equation}\label{circlemmaB}
p(Y)^\circ=p(Y),\qquad p\in\C[T]^{W_0}
\end{equation}
with ${}^\circ\colon H(\underline{k}^{-1})\to H(\underline{k})$ as
defined in Lemma \ref{lemrhodualityB}. We elaborate on the proof
of \eqref{circlemmaB}, which is different than for
$\textup{GL}_N$. Note that since $p\in\C[T]^{W_0}$, the result
follows if we can prove that
$(Y^\lambda)^\circ=T_{w_0}Y^{w_0(\lambda)}T_{w_0}^{-1}$ for
$\lambda\in P^\vee$. Moreover, it suffices to show this only for
specific elements of $P^\vee$, as we demonstrate first. For any
$\lambda\in P^\vee$, let $v_\lambda$ be the shortest element of
$W_0$ such that $v_\lambda(\lambda)=w_0(\lambda)$, and put
$u_\lambda:=\textup{t}(\lambda)v_\lambda^{-1}$. Then by
\cite[(2.5.4)]{M} $\Omega=\{e\}\cup\{u_{\varpi_j^\vee}\}_{j\in J}$
with $J:=\{i\in
1,\ldots,N\mid\langle\varpi_i^\vee,\phi\rangle=1\}$. If
$\lambda\in P^\vee\setminus Q^\vee$, we can write
$\textup{t}(\lambda)=u_{\varpi_j^\vee}w$ for some $j\in J$ and
$w\in W_{Q^\vee}$ (using $W=\Omega\ltimes W_{Q^\vee}$), and then
$\textup{t}(\lambda)=\textup{t}(\varpi_j^\vee)v_{\varpi_j^\vee}^{-1}w=
\textup{t}(\varpi_j^\vee)\textup{t}(\alpha)w'$ for some $\alpha\in
Q^\vee$ and $w'\in W_0$ (using $W_{Q^\vee}=Q^\vee\rtimes W_0$).
But then $w'=e$ and $\lambda=\varpi_j^\vee+\alpha$. In particular,
$\{0\}\cup\{\varpi_j^\vee\}_{j\in J}$ is a complete set of
representatives of $P^\vee/Q^\vee$. Since
$Q^\vee=\textup{span}_\Z\{w(\phi^\vee)\mid w\in W_0\}$, it thus
suffices to show
$(Y^\lambda)^\circ=T_{w_0}Y^{w_0(\lambda)}T_{w_0}^{-1}$ only for
$\lambda=\varpi_j^\vee$ with $j\in J$ and for
$\lambda=w(\phi^\vee)$ ($w\in W_0$).

Let $j\in J$ and write $u_j:=u_{\varpi_j^\vee}$ and
$v_j:=v_{\varpi_j^\vee}$. By \cite[(3.3.3)]{M}, we have
$u_j=T_wY^{w^{-1}(\varpi_j^\vee)} T^{-1}_{v_jw}$ for all $w\in
W_0$. Let ${}^\bullet\colon H(\underline{k})\to
H(\underline{k}^{-1})$ denote the inverse of ${}^\circ$. It
follows that
\[
\begin{split}
(Y^{w_0(\varpi_j^\vee)})^\bullet&=(T_{w_0}^{-1}u_j
T_{v_j w_0})^\bullet=T_{w_0}u_j T_{w_0 v^{-1}_j}^{-1}\\
&=T_{w_0}u_j(T_{w_0}T_{v_j}^{-1})^{-1}=T_{w_0}u_j T_{v_j}T_{w_0}^{-1}\\
&=T_{w_0}Y^{\varpi_j^\vee}T_{w_0}^{-1},
\end{split}
\]
since
$u_jT_{v_j}=T_{u_jv_j}=T_{\textup{t}(\varpi_j^\vee)}=Y^{\varpi_j^\vee}$.
Hence
$(Y^{\varpi_j^\vee})^\circ=T_{w_0}Y^{w_0(\varpi_j^\vee)}T_{w_0}^{-1}$.
Similarly, we can use \cite[(3.3.6)]{M} to obtain
$(Y^{w(\phi^\vee)})^\circ=T_{w_0}Y^{w_0w(\phi^{\vee})}T_{w_0}^{-1}$
for $w\in W_0$, and the proof is complete.
\end{proof}

In order to give more explicit formulas for $L_p^x$ and $L_p^y$,
we need to introduce some notation. For $\lambda\in P^\vee$, write
$W_{0,\lambda}$ for the isotropy subgroup of $\lambda$ in $W_0$,
and $W_0^\lambda$ for a complete set of representatives of
$W_0/W_{0,\lambda}$. We may assume that $e\in W_0^\lambda$. Let
$m_\lambda\in\C[T]^{W_0}$ be the associated monomial symmetric
function, that is, $m_\lambda(t):=\sum_{\mu\in W_0\lambda}t^\mu$.
Finally, set $\Sigma^0(\lambda):=\Sigma(\lambda)-W_0\lambda$
(recall that $\Sigma(\lambda)$ is the smallest saturated subset of
$P^\vee$ that contains $\lambda$, cf. subsection
\ref{subsecFormalB}).

Now fix $\lambda\in P^\vee_-=-P_+^\vee$. By \cite[(4.4.12)]{M}, we
have for $f\in\mathbb{K}$
\begin{equation}\label{LxFormulaB}
\begin{split}
(L^x_{m_\lambda}f)(t,\gamma)=&\sum_{w\in W_0^\lambda} \prod_{a\in
S(\textup{t}(-\lambda))}c_{w(a),\underline{k},q}
(t^{-1})f(q^{-w(\lambda)}t,\gamma)\:+\\
&\sum_{\mu\in\Sigma^0(\lambda)}g_\mu(t)f(q^{-\mu}t,\gamma)
\end{split}
\end{equation}
for some $g_\mu\in\mathcal{M}(T)$ (here we used \eqref{cParamRelB}). In
view of \eqref{LdualityB}, one immediately obtains a similar
formula for $L_{m_\lambda}^y$.

\begin{rema}
For $\lambda=w_0(\varpi_j^\vee)$ with $\varpi_j^\vee$ minuscule
(that is, $\langle\varpi_j^\vee,\alpha\rangle\in\{0,1\}$ for all
$\alpha\in R_+$), we have $\Sigma^0(\lambda)=\emptyset$, while for
$\lambda=-\phi^\vee$ we have $\Sigma^0(\lambda)=\{0\}$. In both
cases one obtains an explicit formula for $L_{m_\lambda}^x$ and
the resulting operators are the Macdonald $q$-difference operators
\cite{M5}.
\end{rema}

We now define the following bispectral version of Macdonald's
eigenvalue problem.
\begin{defi}
We define \textup{BiSP} as the set of solutions $f\in\mathbb{K}$ of the
following bispectral problem:
\begin{equation}\label{BiSPB}
\begin{split}
(L_p^xf)(t,\gamma)&=p(\gamma^{-1})f(t,\gamma),\qquad\forall p\in\C[T]^{W_0},\\
(L_p^yf)(t,\gamma)&=p(t)f(t,\gamma),\qquad\quad\forall p\in\C[T]^{W_0}.
\end{split}
\end{equation}
\end{defi}

\begin{rema}\label{remBiSPB}
Note that $\textup{BiSP}$ is a $\mathbb{W}_0$-invariant
$\mathbb{F}$-linear subspace of $\mathbb{K}$.
\end{rema}

\subsection{The correspondence}
Consider the linear map $\chi_+\colon H_0\to\C$ defined by
$\chi_+(T_w)=k(w)$. By $\mathbb{K}$-linear extension we obtain a
$\mathbb{K}$-linear map $\chi_+\colon
H_0^\mathbb{K}\to\mathbb{K}$. It gives rise to the following
correspondence between $\textup{SOL}$ and $\textup{BiSP}$.
\begin{thm}\label{CherMatB}
The $\mathbb{K}$-linear functional $\chi_+\colon H_0^\mathbb{K}\to\mathbb{K}$
restricts to an injective $\mathbb{W}_0$-equivariant
$\mathbb{F}$-linear map
\[
\chi_+\colon\textup{SOL}\to\textup{BiSP}.
\]
\end{thm}
The theorem follows by restricting Cherednik's correspondence
mentioned in the introduction of this section (for $M$ the formal
principal series module) to $\textup{SOL}$. Indeed, if
$f\in\textup{SOL}$, then for fixed $\gamma\in T$, $f(t,\gamma)$
can be viewed as a solution of qKZ for the $H$-module $M_\gamma$,
and then by Cherednik's correspondence  $\chi_+(f)$ satisfies the
first system of equations of \eqref{BiSPB}. This holds for all
$\gamma\in T$. By \eqref{LdualityB} and the $\iota$-invariance of
\textup{SOL}, it then follows that
\[(L_p^yf)(t,\gamma)=(\iota L_p^x\iota f)(t,\gamma)=
(L_p^x\iota f)(\gamma^{-1},t^{-1})=p(t)(\iota f)(\gamma^{-1},t^{-1})=
p(t)f(t,\gamma),
\]
so also the second equation of \eqref{BiSPB} is satisfied.

For $\textup{GL}_N$, a detailed proof can be found in \cite[\S6]{MS}
and the arguments used there can also be applied in the present
setting.

\section{Harish-Chandra series solutions}

Application of $\chi_+$ to the basic asymptotic solution $\Phi$
leads to a meromorphic solution $\Phi_+$ of the bispectral
problem, which can be viewed as a bispectral analogue of the
difference Harish-Chandra solutions of the Macdonald difference
equations (\cite{LS}). For root systems of type $A$,
Harish-Chandra series solutions were studied before in \cite{EK1}
and \cite{KK}. In \cite[\S6.4]{MS}, the Harish-Chandra series
solution of type $A$ was reobtained from $\Phi_+(t,\gamma)$, by
specializing $\gamma\in T$, yielding new results on the
convergence and singularities of these solutions as a consequence
of corresponding results for $\Phi$. In the final subsection we
extend this to arbitrary root systems.

\subsection{Bispectral Harish-Chandra series}
As announced, we apply the map $\chi_+$ to the basic
asymptotically free solution $\Phi$ of BqKZ to obtain a special
meromorphic solution of the bispectral problem (see
\cite[\S6.3]{MS} for $\textup{GL}_N$).

\begin{defi}
We call $\Phi^+:=\chi_+(\Phi)\in\textup{BiSP}$ the basic Harish-Chandra
series solution of the bispectral problem.
\end{defi}

Put $\Psi^+:=\chi_+(\Psi)$. Then $\Phi^+=G\Psi^+$ and as a
consequence of Proposition \ref{PsiAnalyticB} and Corollary
\ref{PsiGammaExpansionB}, $\Psi^+$ is analytic on
$T\setminus\mathcal{S}_+^{-1}\times T\setminus\mathcal{S}_+$, and
for $(t,\gamma)\in B_{\epsilon}^{-1}\times
T\setminus\mathcal{S}_+$ we have
\[
\Psi^+(t,\gamma)=\sum_{\alpha\in Q^\vee_+}\Gamma^+_\alpha(\gamma)t^{-\alpha},
\]
where $\Gamma^+_\alpha:=\chi_+(\Gamma_\alpha)\in\mathcal{M}(T)$
for all $\alpha\in Q_+^\vee$. Recall that
$\Gamma_0(\gamma)=K(\gamma)T_{w_0}$ for some $K\in\mathcal{M}(T)$
(see \eqref{gamma0B}).

\begin{thm}\label{thmKgammaB}
We have
\begin{equation}\label{GammaPlusB}
\Gamma_0^+(\gamma)=k(w_0)K(\gamma),
\end{equation}
with $K\in\mathcal{M}(T)$ given by
\begin{equation}\label{KexplB}
K(\gamma)=\prod_{\alpha\in
R_+}\frac{(q_\alpha\gamma^{\alpha^\vee};q_\alpha)_\infty}
{(q_\alpha k_\alpha^2\gamma^{\alpha^\vee};q_\alpha)_\infty},
\end{equation}
where $q_\alpha=q^{2/\|\alpha\|^2}$ for $\alpha\in R_+$, as
before.
\end{thm}
\begin{proof}
The definition of $\chi_+$ and the preceding remarks imply
\eqref{GammaPlusB}. Let $L(\gamma)$ denote the right-hand side of
\eqref{KexplB}. Then $L\in\mathcal{M}(T)$ is uniquely
characterized by the following properties.
\begin{itemize}
\item[\bf{(i)}] There exists an $\epsilon>0$ such that for $\gamma\in B_\epsilon$,
 $L$ admits a power series expansion
    \[ L(\gamma)=\sum_{\alpha\in Q^\vee_+}l_\alpha\gamma^\alpha,\]
    converging normally on compacta of $B_\epsilon$.
\item[\bf{(ii)}] $l_0=1$.
\item[\bf{(iii)}] $L(\gamma)$ satisfies the following system of $q$-difference
 equations:
    \[
    \left(\prod_{\alpha\in R_+}
    \prod_{r=1}^{\langle\lambda,\alpha\rangle}
    \frac{1-q_\alpha^r\gamma^{\alpha^\vee}}
    {1-q_\alpha^r k_\alpha^2\gamma^{\alpha^\vee}}\right)
    L(q^{\lambda}\gamma)=L(\gamma),\qquad \lambda\in P^\vee_+.
    \]
\end{itemize}
From Theorem \ref{asymTHMB} it follows that $K$ satisfies (i), and
since $K_{0,0}=T_{w_0}$, $K$ also satisfies (ii). It thus suffices to
show that $K$ solves the $q$-difference equations in (iii).

Recall that in order to show that
$\Gamma_0(\gamma)=K(\gamma)T_{w_0}$ for some $K\in\mathcal{M}(T)$,
we exploited the fact that $\Phi$ is a solution of the quantum KZ
equation in $t$ and investigated what taking the limit
$|t^{-\alpha^\vee_j}|\to0$ had to mean for $\Gamma_0(\gamma)$. We
are now going to exploit the fact that $\Phi^+$ satisfies the
spectral problem
\begin{equation}\label{specPhi+B}
(L_p^y\Phi^+)(t,\gamma)=p(t)\Phi^+(t,\gamma),\qquad p\in\C[T]^{W_0},
\end{equation}
and consider the limit $|t^{-\alpha^\vee_i}|\to0$ to obtain the
desired $q$-difference equations for $\Gamma^+_0$, and hence for
$K$.

Fix $\lambda\in P^\vee_-$. From formula \eqref{LxFormulaB} we
deduce
\[
\begin{split}
(L^y_{m_\lambda}\Phi^+)(t,\gamma)&=\sum_{w\in W_0^\lambda}
\prod_{a\in
S(\textup{t}(-\lambda))}c_{w(a),\underline{k},q^{-1}}
(\gamma)\Phi^+(t,q^{w(\lambda)}\gamma)+
\sum_{\mu\in\Sigma^0(\lambda)}g_\mu(\gamma^{-1})\Phi^+(t,q^\mu\gamma)
\end{split}
\]
with $g_\mu\in\mathcal{M}(T)$. Plugging
in $\Phi^+=G\Psi^+$, using \eqref{WeqB} and dividing both sides by
$G(t,\gamma)$, the equality $(L^y_{m_\lambda}\Phi^+)(t,\gamma)=
m_{\lambda}(t)\Phi^+(t,\gamma)$ gives
\[
\begin{split}
m_\lambda(t)\Psi^+(t,\gamma)=&\sum_{w\in W_0^\lambda}\prod_{a\in
S(\textup{t} (-\lambda))}c_{w(a),\underline{k}}(\gamma)
\delta_{\underline{k}}^{-w(\lambda)}t^{w_0w(\lambda)}
\Psi^+(t,q^{w(\lambda)}\gamma)\:+\\
&\sum_{\mu\in\Sigma^0(\lambda)}g_\mu(\gamma^{-1})
\delta_{\underline{k}}^{-\mu}
t^{w_0(\mu)}\Psi^+(t,q^\mu\gamma).
\end{split}
\]
Now we multiply both sides by $t^{-w_0(\lambda)}$ and consider the
limit $|t^{-\alpha^\vee_j}|\to0$. By \eqref{GammaPlusB} this will
result in a $q$-difference equation for $K$.  Note that:
\begin{itemize}
\item[(1)] $t^{-w_0(\lambda)}m_{\lambda}(t)=\sum_{\mu\in W_0\lambda}
t^{-w_0(\lambda)+\mu}\to 1$ since $w_0(\lambda)\in P^\vee_+$ and
$\nu-w(\nu)\in Q_+^\vee$ for all $\nu\in P_+^\vee$ and $w\in W_0$.
\item[(2)] $t^{-w_0(\lambda)}t^{w_0w(\lambda)}=t^{-w_0(\lambda)+w_0w(\lambda)}$
which is equal to 1 if $w(\lambda)=\lambda$ and tends to 0
otherwise. Considering $w\in W_0^\lambda$, we have
$w(\lambda)=\lambda$ only for $w=e$ .
\item[(3)] $t^{-w_0(\lambda)}t^{w_0(\mu)}\to 0$ for all
$\mu\in\Sigma^0(\lambda)$. Indeed, by \cite[(2.6.3)]{M} we have
\[\mu_+\in\Sigma(w_0(\lambda))\Leftrightarrow w_0(\lambda)-\mu_+\in Q_+^\vee\]
and hence also $w_0(\lambda)-w_0(\mu)\in Q_+^\vee$ for
$\mu\in\Sigma^0(\lambda)\subset\Sigma(w_0(\lambda))$. Moreover,
$w_0(\lambda)\neq w_0(\mu)$ since $\mu\notin W_0\lambda$.
\end{itemize}
Consequently, $K$ satisfies the following set of $q$-difference
equations:
\[
\left(\prod_{a\in
S(\textup{t}(-\lambda))}c_{a;\underline{k},q^{-1}}
(\gamma)\right)\delta_{\underline{k}}^{-\lambda}
K(q^\lambda\gamma)=K(\gamma),\qquad \lambda\in P^\vee_-.
\]
Equivalently, also setting $\mu:=-\lambda\in P_+^\vee$,
\begin{equation}\label{KcompB}
\left(\prod_{a\in
S(\textup{t}(\mu))}\frac{k_a^{-1}-k_a(q^{\mu}\gamma)^{a^\vee}}
{1-(q^{\mu}\gamma)^{a^\vee}}\right) \delta_{\underline{k}}^{\mu}
K(\gamma)=K(q^{\mu}\gamma),\qquad \mu\in P^\vee_+.
\end{equation}
Note that $L^y_{m_\lambda}\in\C(T)\#_{q^{-1}}W\simeq
\C(\{1\}\times T)\#(\{e\}\times W)$, so
$\gamma^{(\alpha+rc)^{\vee}}=q_\alpha^{-r}\gamma^{\alpha^\vee}$ for
$\alpha\in R$ and $r\in\Z$. Using
\[
\begin{split}
\prod_{a\in
S(\textup{t}(\mu))}\frac{k_a^{-1}-k_a(q^{\mu}\gamma)^{a^\vee}}
{1-(q^{\mu}\gamma)^{a^\vee}}&=\prod_{\alpha\in R_+}
\prod_{r=0}^{\langle\mu,\alpha\rangle-1}
\frac{k_\alpha^{-1}-k_\alpha q_\alpha^{\langle\mu,\alpha\rangle}
q_\alpha^{-r}\gamma^{\alpha^\vee}}
{1-q_\alpha^{\langle\mu,\alpha\rangle}q_\alpha^{-r}\gamma^{\alpha^\vee}}\\
&=\prod_{\alpha\in R_+}\prod_{r=1}^{\langle\mu,\alpha\rangle}
\frac{k_\alpha^{-1}-k_\alpha q_\alpha^{r}\gamma^{\alpha^\vee}}
{1-q_\alpha^{r}\gamma^{\alpha^\vee}}
\end{split}
\]
and $\delta_{\underline{k}}^{\mu}= \prod_{\alpha\in
R_+}k_\alpha^{\langle\mu,\alpha\rangle}$, we obtain from \eqref{KcompB}
\[
\left(\prod_{\alpha\in R_+}\prod_{r=1}^{\langle\mu,\alpha\rangle}
\frac{1-k^2_\alpha q_\alpha^{r}\gamma^{\alpha^\vee}}
{1-q_\alpha^{r}\gamma^{\alpha^\vee}}\right)K(\gamma)=
K(q^{\mu}\gamma),\qquad \mu\in P^\vee_+,
\]
and the proof is complete.
\end{proof}

In view of Remark \ref{remBiSPB}, we obtain solutions
$\Phi^+_w\in\textup{BiSP}$ ($w\in W_0$), given by
\[
\Phi^+_w(t,\gamma):=\Phi^+(t,w^{-1}\gamma).
\]
Setting $\Psi^+_w(t,\gamma):=\Psi^+(t,w^{-1}\gamma)$, we have
$\Phi^+_w(t,\gamma)=G(t,w^{-1}\gamma)\Psi_w^+(t,\gamma)$ and by
Corollory \ref{corPsiSpecB}(ii), for $\epsilon>0$ sufficiently
small, $\Psi_w^+$ has a power series expansion
\[
\Psi^+_w(t,\gamma)=\sum_{\alpha\in
Q_+^\vee}\Gamma_\alpha^+(w^{-1}\gamma)t^{-\alpha}
\]
for $(t,\gamma)\in B_{\epsilon}\times T\setminus
w(\mathcal{S}_+)$, converging normally on compacta of
$B_{\epsilon}\times T\setminus w(\mathcal{S}_+)$. The next result
follows along the same line as \cite[Prop. 6.20]{MS}.
\begin{prop}
The set $\{\Phi_w^+\}_{w\in W_0}\subset\textup{BiSP}$ is
$\mathbb{F}$-linearly independent.
\end{prop}
We expect that the set $\{\Phi_w^+\}_{w\in W_0}$ is in fact a
basis of $\textup{BiSP}$ over $\mathbb{F}$. This would follow, for
example, if we could prove that $\chi_+$ is an $\mathbb{F}$-linear
isomorphism $\textup{SOL}\to\textup{BiSP}$ (rather than only an
embedding). Both are still open problems.

\subsection{Application to Harish-Chandra series solutions of
Macdonald's difference equations} Let $\zeta\in T$. The spectral
problem of the Macdonald $q$-difference operators with spectral
parameter $\zeta$ is \begin{equation}\label{MacSpecB}
L_p^xf=p(\zeta^{-1})f,\qquad \forall p\in\C[T]^{W_0},
\end{equation}
for meromorphic functions $f$ on $T$. Let
$\textup{SP}_\zeta\subset \mathcal{M}(T)$ denote the set of
solutions of \eqref{MacSpecB}. It is a vector space over
$\mathcal{E}(T)$, invariant under the usual action of $W_0$ on
$M(T)$.

Recall the solution space $\textup{SOL}_\zeta\subset
H_0^{\mathcal{M}(T)}$ of the quantum KZ equation \eqref{qKZconcrB}
associated to $M_\zeta$, also $W_0$-invariant, but with respect to
the $\tau_x^{M_\zeta}(W_0)$-action on $H_0^{\mathcal{M}(T)}$. We
have the following special case of Cherednik's correspondence from
\cite{Cref,CInd} (see \cite[Prop. 6.22]{MS}).

\begin{prop}
For each $\zeta\in T$, $\chi_+$ defines an $W_0$-equivariant
$\mathcal{E}(T)$-linear map
$\chi_+\colon\textup{SOL}_\zeta\to\textup{SP}_\zeta$.
\end{prop}
\begin{rema}
In an upcoming paper by Stokman it is shown that $\chi_+$ is an
isomorphism if $\zeta^{\alpha^\vee}\neq k^2_\alpha,1$ for all
$\alpha\in R$ (see \cite{St}).
\end{rema}

Recall that $\Psi^+=\chi_+(\Psi)$ with $\Psi$, as usual, the
solution of the gauged bispectral BqKZ equations
\eqref{gaugedeqnB} obtained in Theorem \eqref{asymTHMB}. It
follows from Corollary \ref{corPsiSpecB} that $\Psi^+(t,\gamma)$
may be specialized at $\gamma=\zeta$ for $\zeta\in T\setminus
\mathcal{S}_+^{-1}$, yielding a meromorphic function
$\Psi^+(\cdot,\zeta)\in\mathcal{M}(T)$ with poles at
$t\in\mathcal{S}_+^{-1}$. Define $\widetilde{G}\in\mathbb{K}$ by
\[
\widetilde{G}(t,\gamma):=\frac{\theta_q(tw_0(\gamma)^{-1})}
{\theta_q(\delta_{\underline{k}}t)}.
\]
\begin{rema}
Note that
$\widetilde{G}(t,\gamma)=\theta_q(\delta_{\underline{k}}^{-1}w_0(\gamma)^{-1})
G(t,\gamma)$ and that $\widetilde{G}(t,\gamma)$ can be specialized
in $\gamma=\zeta$. Lacking the factor
$\theta_q(\delta_{\underline{k}}^{-1}w_0(\gamma)^{-1})$ in the
denominator, $\widetilde{G}$ does not satisfy
$\iota(\widetilde{G})=\widetilde{G}$. Therefore,
$\widetilde{G}\Psi\notin \textup{SOL}$, but we do have
$\widetilde{G}(\cdot,\zeta)\Psi(\cdot,\zeta)\in\textup{SOL}_\zeta$.
\end{rema}
It follows that
$\widetilde{G}(\cdot,\zeta)\Psi^+(\cdot,\zeta)\in\textup{SP}_\zeta$
and hence $\Psi^+(\cdot,\zeta)$ is a solution of the spectral
problem for the gauged Macdonald $q$-difference operators with
spectral parameter $\zeta$, that is, a solution of
\begin{equation}\label{gaugedMacB}
(\widetilde{L}^x_pf)(t)=p(\zeta^{-1})f(t),\qquad \forall p\in\C[T]^{W_0},
\end{equation}
with
\[
\widetilde{L}^x_p:=\widetilde{G}(\cdot,\zeta)^{-1}\:L^x_p\:
\widetilde{G}(\cdot,\zeta).
\]
At the end of the previous subsection we introduced
$\Psi^+_w(t,\gamma)=\Psi^+(t,w^{-1}\gamma)$ for $w\in W_0$. Put
$\mathcal{S}:=\bigcup_{w\in W_0}w(\mathcal{S}_+)$. The
considerations of this section imply the following.
\begin{thm}
Fix $\zeta\in T\setminus{S}$.\\
{\bf (i)} For $\epsilon>0$ sufficiently
small, $\Psi_w^+(\cdot,\zeta)$ has a power series expansion
\[
\Psi^+_w(t,\zeta)=\sum_{\alpha\in
Q_+^\vee}\Gamma_\alpha^+(w^{-1}\zeta)t^{-\alpha}
\]
for $t\in B_{\epsilon}$, converging normally on compacta of
$B_{\epsilon}$ and with $\Gamma_0^+(w^{-1}\zeta)\neq0$ explicitly
given by \eqref{GammaPlusB}.\\
{\bf(ii)} $\Psi_w^+(t,\zeta)$ ($w\in W_0$) is analytic in
$t\in T\setminus\mathcal{S}_+^{-1}$.\\
{\bf (iii)} The function $\widetilde{\Psi}^+_w(\cdot,\zeta)\in\mathcal{M}(T)$
($w\in W_0$) defined by
\[
\widetilde{\Psi}^+_w(t,\zeta):=
\frac{\widetilde{G}(t,w^{-1}(\zeta))}{\widetilde{G}(t,\zeta)}
\Psi_w^+(t,\zeta)
=\frac{\theta_q(t(w_0w^{-1})(\zeta)^{-1})}
{\theta_q(tw_0(\zeta)^{-1})}\Psi_w^+(t,\zeta),
\]
is a nonzero solution of the spectral problem \eqref{gaugedMacB}
for the gauged Macdonald $q$-difference operators for all $w\in
W_0$.
\end{thm}
The functions $\widetilde{\Psi}^+_w(\cdot,\zeta)$ ($w\in W_0$) are
the Harish-Chandra series solutions of the spectral problem
\eqref{gaugedMacB}. As already mentioned in the introduction of
this section, formal Harish-Chandra series solutions of
Macdonald's spectral problem were already obtained in \cite{LS},
and earlier for the root system of type $A$ in \cite{EK1} and
\cite{KK}. The upshot here is that we obtain the Harish-Chandra
series solutions as meromorphic functions and are able to
explicitly determine the leading term and the pole locations of $\Psi_w^+(\cdot,\zeta)$.


\end{document}